\documentclass{imsart}
	
	\startlocaldefs
	\usepackage{amsmath, amsthm, amsfonts, amssymb,bbm,dsfont,mathrsfs,yfonts}
\usepackage{boites} 
\usepackage{graphicx,float}
\usepackage{mathalfa}
\usepackage[pagebackref=true]{hyperref,xcolor}

\usepackage{pagecolor}
\usepackage{enumerate}

	\newcommand{\td}{\widetilde}
	\renewcommand{\Re}{{\rm Re\ }}
	
	\newcommand{\maxx}{\vee}
	\newcommand{\minn}{\wedge}
	\usepackage{lipsum}
	\renewcommand{\phi}{\varphi}
	
	\newcommand\numberthis{\addtocounter{equation}{1}\tag{\theequation}}

	\newcommand{\rt}[1]{\sqrt #1}
	
	\newcommand\Item[1][]{%
	  \ifx\relax#1\relax  \item \else \item[#1] \fi
	  \abovedisplayskip=0pt\abovedisplayshortskip=0pt~\vspace*{-\baselineskip}}

\newcommand{\scD}{{\mathscr D}}

\newcommand{\cB}{{\mathcal B}}

\newcommand{\cF}{{\mathcal F}}

\newcommand{\cR}{{\mathcal R}}

\newcommand{\ba}{\begin{eqnarray}}
\newcommand{\na}{\end{eqnarray}}
\newcommand{\ban}{\begin{eqnarray*}}
	\newcommand{\nan}{\end{eqnarray*}}

\newcommand{\C}{{\mathbb C}}
\newcommand{\bD}{{\mathbb D}}

\newcommand{\N}{\mathbb N}
\newcommand{\F}{\mathbb F}
\newcommand{\Em}{{\mathbb E}}

\newcommand{\Pm}{{\mathbb P}}
\newcommand{\Rm}{{\mathbb R}}
\newcommand{\R}{{\mathbb R}}

\newcommand{\Var}{\mathrm{Var}}
\newcommand{\qand}{\quad\mathrm{and}\quad}

\newcommand{\Cov}{\mathrm{Cov}}

\newcommand{\norm}[1]{\lVert {#1} \rVert}
\newcommand{\normm}[1]{\left\lVert {#1} \right\rVert}

\newcommand{\<}{\langle}
\renewcommand{\>}{\rangle}

\newcommand{\ind}{\mathds{1}}
\newcommand{\indd}[1]{\mathds{1}_{\{#1\}}}

	\usepackage[numbers]{natbib}
	\usepackage{fullpage}
	\theoremstyle{plain}
	\newtheorem{theorem}{Theorem}[section]
	\newtheorem{lemma}[theorem]{Lemma}
	\newtheorem{proposition}[theorem]{Proposition}
	\newtheorem{corollary}[theorem]{Corollary}

	\newtheorem{assumption}{Assumptions}[section]

	\theoremstyle{definition}
	\newtheorem{definition}[theorem]{Definition}

	\theoremstyle{remark}
	\newtheorem{remark}{Remark}[section]
	\newtheorem*{remark*}{Remark}

	\numberwithin{equation}{section}

	\newcommand{\loc}{_{\rm loc}}
	
	\newcommand{\pf}[1]{\noindent{\text{#1}}}
	
	\newcommand{\map}{\mapsto}
	\renewcommand{\rt}[1]{\sqrt{#1}}

	\newcommand{\cS}{\mathcal S}

	\makeatletter
\newcommand*\bigcdot{\mathpalette\bigcdot@{.5}}
\newcommand*\bigcdot@[2]{\mathbin{\vcenter{\hbox{\scalebox{#2}{$\m@th#1\bullet$}}}}}
\makeatother
	\endlocaldefs

\begin{document}

	\begin{frontmatter}
		\title{A Continuous Time GARCH(p,q) Process with Delay}
		\runtitle{A Continuous Time GARCH(p,q) Process with Delay}
		
		\begin{aug}	
		\author{\fnms{Adam}  \snm{Nie}\corref{}\thanksref{t2}\ead[label=e1]{adam.nie@anu.edu.au}}
		\thankstext{t2}{RSFAS, ANU College of Business and Economics. Email:adam.nie@anu.edu.au }

		\runauthor{Adam Nie}

		\affiliation{The Australian National University}

		\address{RSFAS, CBE, The Australian National University,\\ 
		    \printead{e1}}
		\end{aug}

		\begin{abstract}	

	We investigate the properties of a continuous time GARCH process as the solution to a L\'evy driven stochastic functional integral equation. 
	This process occurs as a weak limit of a sequence of discrete time GARCH processes as the time between observations converges to zero and the number of lags grows to infinity. The resulting limit generalizes the COGARCH process and can be interpreted as a COGARCH process with higher orders of lags.

	We give conditions for the existence, uniqueness and regularity of the solution to the integral equation, and derive a more conventional representation of the process in terms of a stochastic delayed differential equation. Path properties of the volatility process, including piecewise differentiability and positivity, are studied,  as well as second order properties of the process, such as uniform $L^1$ and $L^2$ bounds, mean stationarity and asymptotic covariance stationarity.
		\end{abstract}	

		\begin{keyword}[class=MSC]
			60G51, 60H10, 60H20, 62M10 
		\end{keyword}

		\begin{keyword}
			L\'evy process; COGARCH process; continuous time GARCH process;  stochastic functional differential equation; 
		\end{keyword}

	\end{frontmatter}

\section{Introduction} \label{section/1/Introduction}
	The ARCH (autoregressive conditionally heteroscedasticity) and GARCH (generalized ARCH) models, introduced by \citet{Engle1982} and \citet{Bollerslev1986}, are widely used in financial econometrics to capture stylized features observed in asset return time series, including heteroscedasticity, volatility clustering and heavy tailedness.
	These classes of models aim to capture the conspicuous dependence between asset returns and their volatility, whereby a large fluctuation in the asset price typically causes a large fluctuation in the volatility, which persists for a period of time before reverting to a baseline level.
	For a review of the probabilistic properties, stationarity, and mixing properties of GARCH models, we refer to \citep{BougerolPicard1992} and \citep{Lindner2009a}. 

	For certain applications including option pricing and the modeling of irregularly spaced or high frequency data, a model in continuous time is often preferred. For a comprehensive review of continuous time GARCH processes, we refer to \citet{Lindner2009}.
	The first notable attempt was due to \citet{Nelson1990}, who constructed a diffusion limit from a suitably scaled sequence of GARCH(1,1) processes, akin to constructing a Wiener process from a sequence of scaled random walks.  
	However, the resulting diffusion process loses desirable properties possessed by the approximating GARCH processes in discrete time.
	In particular, the diffusion limit is driven by two independent Wiener processes whereas the GARCH process is driven by a single sequence of noise, and the feedback mechanism between the price and the volatility is lost in the limit. 
	Furthermore, \citet{Wang2002} shows that parameter estimation for the discrete time GARCH process is not equivalent to that of the diffusion limit.
	Nevertheless, numerous authors have studied the properties and applications of Nelson's limit, and we refer to  \citep{Corradi2000,Duan1997,KallsenTaqqu1998} and the references within.

	\citet*{KluppelbergLindnerMaller2004} took a significantly different approach and constructed the Continuous Time GARCH (COGARCH) process by replacing the innovation sequence in the original GARCH with the increments of a L\'evy process. 
	The resulting variance process is a generalized Ornstein-Uhlenbeck process (see \cite{LindnerMaller2005}) driven by a single L\'evy noise, and retains many of the desirable features of the original GARCH process. 
	\citet{KallsenVesenmayer2009} demonstrated that akin to the construction in \cite{Nelson1990}, the COGARCH process can indeed be obtained as a weak limit of a sequence of GARCH(1,1) processes, and argued heuristically that the diffusion limit and the COGARCH limit are the only possible limits of sequences of GARCH(1,1) processes.
	A different discrete approximation scheme was suggested by \citet*{MallerMullerSzimayer2008}, who provided a pseudo-maximum-likelihood estimator applicable to irregularly spaced data.
	Parameter estimation in the COGARCH process is considered in for instance \citep{HaugKluppelbergLindnerEtAl2007,MallerMullerSzimayer2008,Muller2010} and the COGARCH is applied in option pricing in numerous works including \citep{KluppelbergLindnerMaller2006,KluppelbergMallerSzimayer2010}. An analogous result to \cite{Wang2002} for the Nelson limit was obtained in \citet{BuchmannMuller2012}.

	Both the diffusion limit and the COGARCH process are Markovian processes, and the serial dependence in the variance process is implicit in the defining stochastic differential equation (SDE). 
	\citet{Lorenz2006} on the other hand constructed a non-Markovian continuous GARCH process that explicitly specifies the serial dependence in the volatility process. 
	Unlike the diffusion limit and the COGARCH process which are weak limits of GARCH(1,1) processes, 
	\cite{Lorenz2006} considered a sequence of GARCH$(pn+1,1)$ volatility processes defined sequence of uniform grids of mesh size $n^{-1}$ and took $n$ to infinity.
	A weak limit is obtained as the solution of a stochastic functional differential equation (SFDE) driven by a Wiener process, which is direct generalization of the Nelson limit.

	The same idea was explored in greater depth recently in the PhD thesis of \citet{Tran2013} and the working papers of \citet*{DunsmuirGoldysTran,DunsmuirGoldysTrana}, where the authors considered a more general situation in which both GARCH and ARCH orders are allowed to tend to infinity as the grid becomes finer. 
	The sequence of discrete GARCH price and volatility processes are shown to converge weakly in the Skorokhod topology to the solution of a pair of stochastic functional integral equations
	\begin{align*}
		Y_t & =  Y_0 + \int_0^t \sqrt{X_{s-}} dL_s,\\
		X_t & =  \theta_t + \int_{-p}^0 \int_{u}^{t+u} X_s ds \mu(du)
						+ \int_{-q}^0 \int_{u+}^{t+u} X_{s-} d[L,L]_s \nu(du),
	\end{align*}
	where $Y$ and $X$ are the return and variance processes respectively. Here, $\mu$ and $\nu$ are Borel measures representing the delay effects of higher order lags, $L$ and $\theta$ are semimartingales, and $[L,L]$ is the quadratic variation process of $L$. 
	This limit is referred to as the continuous time GARCH process with delays $p,q\ge 0$, or the CDGARCH$(p,q)$ process for short. 
	Depending on how the discrete noise sequence is constructed, 
	the limiting noise sequence $L$ could be a Brownian motion, in which case the CDGARCH generalizes Lorenz's limit in \cite{Lorenz2006}, or a L\'evy process, of which the COGARCH process is a special case. 
	In this paper, we focus on the latter case of a CDGARCH$(p,q)$ process driven by a L\'evy process, and compare its properties to the COGARCH and similar processes in the literature. 

	When $p>0$ and $q=0$, the CDGARCH$(p,0)$ variance process satisfies an SFDE with an affine (deterministic) drift coefficient and multiplicative noise. Although this class of equations appear in the existing literature, the usual assumptions made are generally too restrictive for our case. 
	Earlier work by \citet*{ReissRiedlevanGaans2006} focuses on the existence of an invariant measure of the solution to this class of SFDEs, while the uniqueness of the invariant distribution is discussed in \citet*{HairerMattinglyScheutzow2011}. 
	In the general case where $q\ne 0$, the resulting SFDE of the volatility process (derived in Section \ref{section/1/solution/existence}) has a non-deterministic drift coefficient that depends on past values of the driving noise as well as the volatility process, a case not considered in most works in the literature. 
	The monograph \citet*{BaoYinYuan2016} collected a series of recent papers, some of which have a close connection to the CDGARCH process and the question of strict stationarity.

	The rest of the paper is organized as follows. Section \ref{section/1/model} outlines the CDGARCH process constructed in \cite{DunsmuirGoldysTran}, motivated its resemblance to the GARCH process in discrete time. Section \ref{section/1/preliminaries} collects some preliminary results on stochastic integration, L\'evy processes, and deterministic  functional differential equations. We present the main results of this paper in Section \ref{section/1/solution} and Section \ref{section/1/moments}, but defer all proofs and technical lemmas to Section \ref{section/1/proofs} in order to maintain the flow of our exposition. 

	Section \ref{section/1/solution/existence} gives conditions for the existence, uniqueness and regularity of a strong solution to the CDGARCH equations, as well as differential representations of the solution in more conventional forms. Section \ref{section/1/solution/pathproperties} studies the path properties of the CDGARCH process under compound Poisson driving noise, while Section \ref{section/1/solution/compoundpoisson} shows that the solution in general can be approximated with solutions driven by compound Poisson processes. Section \ref{section/1/moments/uniform bounds} gives uniform $L^1$ and $L^2$ bounds for the solution process, and Section \ref{section/1/moments/mean function} studies the second order properties of the solution, including mean stationarity and covariance structure of the volatility process as well as the return process.

\section{The CDGARCH Model}\label{section/1/model}
	The GARCH$(P,Q)$ model in discrete time, with GARCH order $P$ and ARCH order $Q$,  is defined in \cite{Bollerslev1986} as the bivariate process 
	\begin{subequations}
	\begin{align}
		Y_n:&= Y_{n-1}+ \sqrt{X_n} Z_n,\quad n\in\N, 			\label{equation/1/model/GARCH_Y}
			\\
		X_n:&= \eta + \sum_{k=1}^P \beta_k X_{n-k} + \sum_{k=1}^Q \alpha_k X_{n-k} Z_{n-k}^2, 
		\quad n\in\N. \label{equation/1/model/GARCH_X}
	\end{align} 
	\end{subequations}
	Here $\eta>0$ and $(Z_n)_{n\in\N}$ is a sequence of uncorrelated random variables with zero mean and unit variance. The non-negative real sequences $(\beta_i)_{1\le i\le P}$ and $(\alpha_i)_{1\le i\le  Q}$ are the GARCH and ARCH parameters respectively. The sequence $(Y_n)_n$ usually represents the log-return of a financial asset, while $X_n=\Var(Y_n|\cF_{n-1})$ models its conditional variance. 

	We now introduce a form of the CDGARCH equation to be studied in this paper, motivated by its resemblance with the GARCH process in discrete time. Writing $\td \beta_k = \beta_k - \ind_{{\{k=1\}}}$, for $P,Q>0$, the GARCH$(P,Q)$ process defined above can be rewritten as 
	\begin{subequations}
	\begin{align*}
		Y_n & =  Y_0 + \sum_{i=1}^n \sqrt{X_i} Z_i,\numberthis\label{equation/1/model/GARCH_aY_naogous}
			\\
		X_n
		 \numberthis\label{equation/1/model/GARCH_X_analogous}
		& =  X_0 + n \eta + \sum_{i=1}^n \sum_{k=1}^P (\beta_k-\indd{k=1}) X_{i-k} 
						+ \sum_{i=1}^n \sum_{k=1}^Q \alpha_k X_{i-k} Z_{i-k}^2
						   \\
		& =  X_0 + n \eta +\sum_{k=-P}^{-1} \td\beta_{-k} \sum_{i=1+k}^{n+k} X_{i} 
						+  \sum_{k=-Q}^{-1} \alpha_{-k} \sum_{i=1+k}^{n+k} X_{i} Z_{i}^2.
	\end{align*}
	\end{subequations}
	This formulation motivated \cite{DunsmuirGoldysTran,DunsmuirGoldysTrana,Tran2013} to study the following analogous setup in continuous time. Fix a filtered probability space $(\Omega, \cF, \mathbb F := (\cF_t)_{t\ge -r}, \Pm)$ that satisfies the usual assumptions (see \cite{Protter2004}, page 3). The CDGARCH$(p,q)$ equation with delays of length $p, q\ge 0$ is defined by the stochastic functional integral equations
	\begin{subequations}
	\begin{align}
		Y_t&= Y_0 + \int_0^t \sqrt{X_{s-}} dL_s, \quad t\ge 0,\label{equation/1/model/CDGARCH_Y}
			\\
		X_t&= \theta_t + \int_{-p}^0 \int_{u}^{t+u}  X_s ds \mu(du)						
		+ \int_{-q}^0 \int_{u+}^{t+u} X_{s-} d[L,L]_s \nu(du),\quad t\ge 0,
		 \label{equation/1/model/CDGARCH_X}
	\end{align}
	\end{subequations}
	with positive initial conditions $Y_0= \Psi$ and $X_u=\Phi_u$ for all $u\in[-(p\maxx q),0]$, where $\Psi$ is $\cF_0$-measurable and $(\Phi_u)_{u\in[-(p\maxx q),0]}$ is adapted. 
	Here, $\theta_t$ is  an adapted c\`adl\`ag process analogous to the term $X_0 + \eta n$ in \eqref{equation/1/model/GARCH_X_analogous}. The measures
	$\mu$ and $\nu$ are signed Borel measures with finite variations, supported on $[-p,0]$ and $[-q,0]$ respectively, analogous to the sequence of coefficients ${(\td \beta_k)}_{1\le k\le P}$ and $(\alpha_k)_{1\le k\le Q}$. 
	The process $L$ is a locally square integrable L\'evy process adapted to $\mathbb F$, 
	analogous to the white noise sequence $(Z_\nu)_n$, and $[L,L]$ is the quadratic variation process of $L$ analogous to the sequence $(Z_n^2)_n$. We will specify $L$ in detail in Section \ref{section/1/preliminaries/levyprocess}.

	In this paper we will assume the following specification of the process $\theta$ and the delay measures $\mu$ and $\nu$. The motivation behind these choices comes from their resemblance to the discrete time GARCH, as well as from the methodology in the construction of the CDGARCH process as a weak limit of discrete GARCH processes in \cite{DunsmuirGoldysTran,DunsmuirGoldysTrana,Tran2013}.

	Put $r:=p\maxx q$. Let $\scD_{[a,b]} := \scD([a,b])$ (resp. $\bD_{[a,b]}$) be the space of c\`adl\`ag functions (resp.  processes) on $[a,b]\subseteq\Rm$ and write $\scD:=\scD_{[-r,0]}$ and $\bD:=\bD_{[-r,0]}$. Given an initial process $\Phi\in\bD$, we extend it to $\bD_{[-r,\infty)}$ by setting $\Phi_t=\Phi_0$, for all $t>0$. 
	Fix a positive constant $\eta$. Throughout the paper we will assume $\theta$ takes the form
	\begin{align}
		\theta_t:= \Phi_t+ \eta t \ind_{[0,\infty)}(t), \quad t\in[-r,\infty),
		\label{equation/1/model/theta}
	\end{align}
	and assume the delay measures have point masses at $0$ and are absolutely continuous with respect to the Lebesgue measure on $[-r,0)$. That is, 
	for any $E\in\cB([-r,0])$, 
	\begin{align}
		\label{equation/1/model/munu}
		\mu(E)&:=
		\int_{E\cap [-p,0]} f_\mu(u) du
		-c_\mu\delta_0(E),
		\qand
		\nu(E):=
		\int_{E\cap [-q,0]} f_\nu(u) du
		+c_\nu\delta_0(E). 
	\end{align}
	Here $c_\mu$ and $c_\nu$ are positive constants, $f_\mu$ and $f_\nu$ are nonnegative and continuous functions supported on $[-p,0]$ and $[-q,0]$ respectively, and $\delta_0$ denotes the Dirac measure at zero. 
	Since $f_\mu$ and $f_\nu$ are supported on $[-p,0]$ and $[-q,0]$, it is clear that $\int_{[-p,0]} f_\mu= \int_{[-r,0]} f_\mu $ and $\int_{[-q,0]} f_\nu= \int_{[-r,0]} f_\nu$. We will write $\int_{[-r,0]}$ for either integral throughout the paper when there is no ambiguity of the meaning.

	\begin{remark}\label{remark/1/model/choicesofmeasures}
		These choices, including the signs of the constants, are very natural in that they arise as the limit of the sequences $(\td \beta_k)_{1\le k\le P}$ and $(\alpha_k)_{1\le k\le Q}$ as a sequence of GARCH processes converges weakly to the CDGARCH limit. On the other hand, we may study the CDGARCH equation as a continuous time GARCH process in its own right without constraining it to be a proper weak limit, then we have more freedom in choosing $\theta$, $\mu$ and $\nu$. \qed
	\end{remark}

	\begin{remark}\label{remark/1/model/COGARCH}
		When $p=q=0$, $c_\mu, c_\nu >0$, and $\theta_t = X_0 + \eta t\ind_{[0,\infty)}(t)$, the CDGARCH variance equation \eqref{equation/1/model/CDGARCH_X} reduces to a stochastic differential equation
			\begin{align*}
				dX_t =  \eta dt -c_\mu X_t dt + c_\nu  X_{t-} d[L,L]_t,\quad t>0,\numberthis\label{equation/1/model/COGARCH}
			\end{align*}
		which (with a reparameterization) is the SDE specifying the COGARCH process (see \cite{KluppelbergLindnerMaller2004} Proposition 3.2). 
		On the other hand, taking $L$ to be a Brownian motion, it is possible to define a similar pair of SFDEs that generalizes Nelson's diffusion and Lorenz's limit. We do not pursue this setup. 
		
	\end{remark}
	
	\begin{remark}
		In the case where $p\ne0$ and $q=0$, by Fubini's Theorem, the variance process follows a stochastic delayed differential equation
			\begin{align*}
				dX_t&= \left(\eta + \int_{[-p,0]} X_{t+u} \mu(du)\right) dt
						+ c_\nu  X_{t-} d[L,L]_t ,\numberthis\label{equation/1/model/CDGARCH(p,0)_X}
			\end{align*}
			which has an affine (delayed) drift term and multiplicative noise. A similar equation was considered in \citet{ReissRiedlevanGaans2006}. Although more general in some respects, \cite{ReissRiedlevanGaans2006} assumed uniform boundedness of the diffusion coefficient, which does not apply in our situation.
	\end{remark}

\section{Preliminaries}\label{section/1/preliminaries}
We first collect some preliminary results. We follow \citet{JacodShiryaev2013,Protter2004} semimartingale theory, \citet{Applebaum2009} for L\'evy processes, and \citet*{DiekmannGilsLunelEtAl2012} for deterministic delay differential equations. 

	\subsection{Driving L\'evy process}\label{section/1/preliminaries/levyprocess}
	Recall $r:=p\maxx q>0$ and suppose we have a filtered probability space $(\Omega, \cF, \F:=(\cF_t)_{t\ge -r}, \Pm)$ that satisfies the ``usual conditions'' (see Definition 12, \cite{JacodShiryaev2013}).  Given a stochastic process $Z$, we write $\sigma(Z):=(\sigma\{Z_u, u\le t\}  )_{t\ge -r}$ for the natural filtration of $Z$. 

	Let $(M_t)_{t\ge -r}$ be a c\`adl\`ag, adapted martingale with respect to $\mathbb F$. We follow \citet{Protter2004} and call $M$ a square integrable martingale if $\Em[M_t^2]<\infty$ for every $t\ge -r$. For a process $Z$ with finite second moments, i.e. $\Em[Z^2_t]<\infty$ for all $t$, 
	write $[Z]:=[Z,Z]$ (resp. $\langle Z \rangle:= \langle Z,Z \rangle$) for the quadratic variation (resp. predictable quadratic variation) process of $Z$. 
	Let $L^2(Z)$ be the set of all predictable processes $H$ such that the integral process $H^2 \cdot \langle Z \rangle$ is integrable, i.e. $\Em[\int_{-r}^T H^2_s d \langle Z\rangle]<\infty$ for each fixed $T$. The following lemma follows from Theorems I.4.31 - I.4.40 of \citet{JacodShiryaev2013}.
	\begin{lemma}\label{lemma/1/preliminaries/Jacod}
		Let $Z$ be a semimartingale and suppose $H$ is c\`adl\`ag and predictable. Then the integral process $H\cdot Z$ is a c\`adl\`ag, adapted process. If furthermore $Z$ is a square integrable martingale and $H\in L^2(Z)$, then $H\cdot Z$ is a square integrable martingale. 
	\end{lemma}

	We assume that the space $(\Omega, \cF, \mathbb F, \Pm)$ supports a c\`adl\`ag, $\mathbb F$-adapted L\'evy process $(L_t)_{t\ge -r}$, such that $L_{-r}=0$ a.s.,  $L$ is stochastically continuous and for all $-r\le s<t<\infty$, $L_t-L_s$ is independent of $\cF_s$ and has the same distribution as $L_{t-s-r}$. 
	Put $\Rm_0 := \Rm\setminus\{0\}$ and write $\cB(\Rm_0)$ for the Borel sigma-algebra on $\Rm_0$. 
	When $U\in \cB(\Rm_0)$ with $0\notin \overline U$, write
	\begin{align*}
		N(t,U): = \sum_{-r\le s\le t} \ind_{U} \left(\Delta L_s\right), \quad t>0
	\end{align*}
	for the Poisson random measure on $\cB(0,\infty)\times \cB(\Rm_0)$ associated with $(L_t)_{t\ge -r}$ and write $\Pi_L(U):= \Em[N(-r+1,U)]$ for the corresponding L\'evy measure on $\cB(\Rm_0)$. 
	Write $\td N(dt,dz) := N(dt, dz) - \Pi_L(dz)dt$ for the compensated Poisson random measure.

	Recall that a L\'evy measure $\Pi_L$ always satisfies $\int_{\Rm_{_0}} (1\minn z^2) \Pi_L(dz)< \infty$. Throughout the paper, we will also assume that $\Pi_L$ has finite second moment and $L$ is centered, so that  $(L_t)_{t\ge -r}$ is a square integrable martingale with respect to $\mathbb F$, i.e., $\Em[L_t] =0$ and  $\Em[L_t^2]<\infty$ for all $t\ge -r$. 
	The Fourier transform of the law of $L_t$ is then given by the L\'evy-Khintchine formula
	\begin{align*}
		\Em\left[e^{iu L_t }\right] 
		& = \exp\left\{ (t+r) \left(- \frac{1}{2}\sigma_L^2 u^2 
		+ \int_{\Rm_0} (e^{iuz} - 1 -iuz)\Pi_L(dz) \right) \right\},\quad u\in\Rm,
	\end{align*}
	where $\sigma_L>0$. 
	Furthermore, the L\'evy-It\^o decomposition of $L$ gives
	\begin{align*}
		L_t = \sigma_L B_{t} + \int_{-r}^{t} \int_{\Rm_0} z \td N(dz, ds), t\ge -r,
		\numberthis\label{equation/1/preliminaries/drivingnoise/L}
	\end{align*}
	where $(B_{t})_{t\ge -r}$ is a standard Brownian motion with respect to $\mathbb F$, having $B_{-r}=0$, a.s. 
	The quadratic variation process $S:=[L,L]$ of $L$ is the subordinator
	\begin{align*}
		S_t = \sigma_L^2 (t+r) + \int_{-r}^t\int_{\Rm_0} z^2 N(dz, dt), \quad t\ge -r. \numberthis\label{equation/1/preliminaries/drivingnoise/S}
	\end{align*}
	Put $\kappa_2: = \Em[S_{-r+1}] = \sigma_L^2 + \int_{\Rm_0} z^2 \Pi_L(dz)<\infty$, so that the process $(\td S_t)_{t\ge -r}$ defined by
	\begin{align*}
		\numberthis \label{equation/1/preliminaries/drivingnoise/tdS}  	
		\td S_t := S_t - \kappa_2(t+r)= \int_{-r}^t\int_{\Rm_0} z^2 \td N(dz, dt), \quad t\ge -r,
	\end{align*}  
	is a martingale with respect to $\mathbb F$ (see \cite{Applebaum2009}, Theorem  2.5.2). 
	If furthermore $L$ has finite fourth moments, then $\kappa_4:=\Em[\td S^2_{-r+1}]<\infty$ and $\td S$ is a square integrable martingale, with predictable quadratic variation process $d\langle S \rangle_t = \kappa_4 dt$.

	\subsection{Delay differential equations}\label{section/1/preliminaries/dde}
	Consider the deterministic functional differential equations
	\begin{align*}
		\frac{d}{dt}x(t) = \int_{[-r,0]} x(t+u) \mu (du), \quad t\ge 0,\numberthis\label{equation/1/preliminaries/dde/dde}
	\end{align*}
	with initial condition $x|_{[-r,0]}=\phi$ for some $\phi\in\scD$. Here $\mu$ is a signed Borel measure with finite total variation on $[-r,0]$. For each initial condition $\phi\in \scD$, there exists a unique  solution $t\mapsto x(t, \phi)$ on $[-r, \infty)$, i.e. $x(u, \phi) = \phi(u)$ for all $u\in [-r,0]$, $t\mapsto x(t, \phi)$  is continuously differentiable on $(0, \infty)$, and \eqref{equation/1/preliminaries/dde/dde} holds on $(0, \infty)$. 
	The asymptotic stability of this solution as $t\to \infty$ is governed by the roots of the so-called \textit{characteristic function} $\Delta:\C\to \C$ of $\mu$, defined as
	\begin{align*}
		\Delta(z) := z - \hat \mu(z) = z-\int_{[-r,0]} e^{zu} \mu(du).
		\numberthis\label{equation/1/preliminaries/dde/characteristicfunction}
	\end{align*}
	Let $x(\cdot, \phi)$ be a solution to \eqref{equation/1/preliminaries/dde/dde} and fix any $\lambda\in\Rm$ such that $\Delta(z)\ne0$ on the line $\Re z =\lambda$. Then \cite{DiekmannGilsLunelEtAl2012} gives the following asymptotic expansion of $t\mapsto x(t, \phi)$:
	\begin{align}
		x(t,\phi) = \sum_{j=1}^n p_j(t) e^{z_j t} + o(e^{\lambda t}), \quad t\to \infty,\label{equation/1/preliminaries/dde/asyExpansion}
	\end{align}
	where $z_1,\ldots, z_n$ are finitely many zeros of $\Delta(z)$ with real part exceeding $\lambda$, and $p_j(t)$ is a $\C$-valued polynomial in $t$ of degree less than the multiplicity of $z_j$ as a zero of $\Delta(z)$. 
	In particular, it's clear from \eqref{equation/1/preliminaries/dde/asyExpansion} that if $\Delta(z)$ is root free in the right half-plane $\{z|\Re z\ge 0\}$, then the zero solution is asymptotically stable, that is, 
	all solutions $x(\cdot, \phi)$ of the functional differential equation \eqref{equation/1/preliminaries/dde/dde} converge to the zero solution exponentially fast as $t\to\infty$. 
	

\section{The solution process}\label{section/1/solution}
We wish to rewrite \eqref{equation/1/model/CDGARCH_X} into a form more commonly seen in the literature, but doing so requires some regularity of the solution $X$. We thus start with establishing the existence, uniqueness and regularity of a strong solution to \eqref{equation/1/model/CDGARCH_X}. All proofs and supporting lemmas are deferred to Section \ref{section/1/proofs} in order to maintain the flow of the exposition. 

\subsection{Existence, uniqueness and representations}\label{section/1/solution/existence} We define the appropriate space for the solution process and give sufficient conditions for the existence of a unique solution in this space.

\begin{definition}
	For $Z\in \bD_{[-r, \infty)}$, define the maximum process $Z^*_t:=\sup_{0\le s\le  t} |Z_s|$ and the random variable $Z^*:=\sup_{s\ge 0} |Z_s|$.
	Let $(\norm{\cdot}_t)_{t\ge -r}$ be a family of semi-norms given by
	\begin{align*}
		\norm{Z}_{t}:=\norm{Z^*_t}_{L^2(\Omega, \Pm)}
		=\left( \Em\left[\sup_{s\in[-r,t]} |Z_s|^2\right]\right)^{1/2}.
		\numberthis\label{equation/1/existence/normt}
	\end{align*}
	We denote by $\cS^2$ the class of c\`adl\`ag processes on $[-r,\infty)$ with finite $\norm{\cdot}_t$ for every $t\ge -r$. 
	\qed
\end{definition}

\begin{definition}\label{definition/1/strongsolution}
	A process $X= (X^\Phi_t)_{t\ge -r}$ is called a strong solution to the variance equation \eqref{equation/1/model/CDGARCH_X} with $\bD$-valued initial condition $\Phi$ if $X$ is in $\cS^2$, is adapted to $\mathbb F$, satisfies $X|_{[-r,0]} = \Phi$, and \eqref{equation/1/model/CDGARCH_X} holds on $(0,\infty)$. We refer to this solution as the CDGARCH$(p,q)$ variance process. \qed
\end{definition}
	\begin{assumption}\label{assumptions/1/existence}
		\begin{enumerate}[(a)]
		\item[] 

		\item The initial process $\Phi\in\bD$ is adapted to $\sigma(L)$, with $\norm{\Phi}_0<\infty$.
		\label{assumptions/1/existence/Phi}

		\item The process $S$ as defined in \eqref{equation/1/preliminaries/drivingnoise/S} is square integrable, i.e. $\Em[L^4_1]<\infty$.
		\label{assumptions/1/existence/S}
		\qed
		
		\end{enumerate}
	\end{assumption} 
	\begin{theorem}\label{theorem/1/existence} 
	Suppose $S$ and $\Phi$ satisfy Assumptions \ref{assumptions/1/existence}.
	\begin{enumerate}[(a)]
		\item There exists a unique strong solution $X$ to \eqref{equation/1/model/CDGARCH_X} with initial condition $\Phi$. 
		\label{theorem/1/existence/1}
		\item For all $\alpha\in[0,2]$, the function $t\mapsto \Em[|X_t|^\alpha]$ is finite valued and c\`adl\`ag.
		\label{theorem/1/existence/2}
		\qed
	\end{enumerate}
	
	\end{theorem}
	\begin{remark}
		At this point we do not require the solution $X$ to be bounded away from zero or even positive. We give a positive lower bound for $X$ in Theorem \ref{theorem/1/solution/positivity} of Section \ref{section/1/solution/compoundpoisson}, after deriving a more convenient representation.
		It is then immediate that the CDGARCH equation \eqref{equation/1/model/CDGARCH_Y} is also well defined and a semimartingale, with jumps given by $\Delta Y_t = X_{t-}^{1/2} \Delta L_t$. \qed
	\end{remark}
	
	We precede Theorem \ref{theorem/1/semimartingale} with the following useful results. Recall that the delay densities $f_\mu$ and $f_\nu$ from \eqref{equation/1/model/munu} are $L^1$ functions. Let $F_\mu, F_\nu:\Rm^2\to \Rm$ be kernels given by
	\begin{align*}
		F_\mu(t,s) & :=   \int_{[-p\vee (s-t), s\wedge 0]} f_\mu(u) du,
		\quad
		F_\nu(t,s)  :=   \int_{[-q\vee (s-t), s\wedge 0]} f_\nu(u) du.
		\numberthis\label{equation/1/semimartingale/kernelF}
	\end{align*}
	Then $F_\mu$ and $F_\nu$ are clearly Volterra type kernels on $\Rm^2$, i.e. $F(t,s)=0$ for all $s\ge t$. 
	It will turn out to be useful to consider the stochastic process $(\Xi(X)_t)_{t\ge 0}$ defined by 
	\begin{align*}
		\Xi(X)_t := \int_{[-p,t]} F_\mu(t,s)  X_{s}   d{s}
		+ \int_{(-q,t]} F_\nu(t,s)  X_{s-}   dS_{s},\quad t\ge 0.
		\numberthis\label{equation/1/semimartingale/Xi(X)VolterraKernel}
	\end{align*}
	
	\begin{proposition} \label{proposition/1/semimartingale}
	Let $f_\mu$ and $f_\nu$ be non-negative and continuous on $[-p,0]$ and $[-q,0]$.
		\begin{enumerate}[(a)]
			\item \label{proposition/1/semimartingale/kernelF}
			The kernels $F_\mu$ and $F_\nu$ are non-negative Lipschitz continuous functions on $\Rm^2$. 
			
			\item \label{proposition/1/semimartingale/Xi} 
			The process $(\Xi(X)_t)_{t\ge 0}$ has locally Lipschitz continuous sample paths, 
			with derivative
			\begin{align*}
			\numberthis\label{equation/1/semimartingale/xi}
			\xi(X)_t:= \frac{d}{dt}\Xi(X)_t = \int_{-p}^0 f_\mu(u)X_{t+u}du + \int_{-q+}^0 f_\nu(u) X_{t+u-}dS_{t+u},
			\end{align*}
			for Lebesgue almost every $t\ge 0$.
			\qed

		\end{enumerate}
	\end{proposition}

	We now give two different representations of $X$, one as a stochastic integral equation with Volterra type kernels, akin to the fractional L\'evy process, and one in the form of a stochastic functional differential equation driven by the quadratic variation process $S$. 
	\begin{theorem}\label{theorem/1/semimartingale}
	Let $X$ be the unique strong solution to the CDGARCH$(p,q)$ variance equation \eqref{equation/1/model/CDGARCH_X}, with parameters specified in \eqref{equation/1/model/theta}, \eqref{equation/1/model/munu} and driving noise $S$ defined in \eqref{equation/1/preliminaries/drivingnoise/S}. 
	Then
		the process $X$ satisfies the stochastic (Volterra) integral equation
		\begin{align*}
			\numberthis\label{equation/1/semimartingale/XwithFkernel}
			X_t   =  X_0 &+ 
			\int_{0}^t \big(\eta- c_\mu X_{s}\big) d{s}
			+ c_\nu\int_{0+}^t X_{s-} dS_{s}
			+ \Xi(X)_t, \quad t\ge 0,
		\end{align*}
		which can be rewritten into a stochastic functional differential equation
		\begin{align*}
			dX_t &  =  \big(\eta -c_\mu X_t+ \xi(X)_t\big) dt + c_\nu X_{t-}  dS_t, \quad t\ge 0,
			\numberthis\label{equation/1/semimartingale/XSDDE}
		\end{align*}
		with initial condition $X_u = \Phi_u$ on $[-r,0]$. In particular, $X$ is a semimartingale and has paths of  finite variation on compacts. \qed
	
	\end{theorem}
	\begin{remark}\label{remark/1/semimartingale/LSSconnection}
		\begin{enumerate}[(a)]
			\item The stochastic process $t\map \Xi(X)_t$ in \eqref{equation/1/semimartingale/Xi(X)VolterraKernel}
			is an example of a \textit{convoluted L\'evy process}, studied in \citet{BenderMarquardt2008}. In fact, with a different choice of kernel, we could recover a fractional L\'evy process considered in \citet*{BenassiCohenIstas2004} and \citet{Marquardt2006}. \citet*{JaberLarssonPulido2017} also considered a similar process with convolution type kernels and Brownian driving noise, with applications to modeling asset volatility.

			\item Without the term $\xi(X)$, the equation \eqref{equation/1/semimartingale/XSDDE} is exactly the SDE for the COGARCH process, stated previously in \eqref{equation/1/model/COGARCH}. We can therefore interpret the CDGARCH process as a COGARCH process with an extra stochastic delay-type drift term $\xi(X)$, which depends on the sample paths of both $X$ and $S$. The CDGARCH process is hence not a Markovian process. 

		\item	
		The process $\xi(X)$ has been studied by \citeauthor{Barndorff-NielsenBenthVeraart2012} in multiple works over the past few years to model stochastic volatility and turbulent flows. In particular $\xi(X)$ is referred to as a {\it volatility modulated L\'evy driven Volterra (VMLV)}, or more specifically, a {\it L\'evy semi-stationary (LSS) process} in \cite{Barndorff-NielsenBenthVeraartEtAl2013}. Furthermore, these processes are special cases of a much more general class of objects called Ambit fields. We refer to \citet*{Barndorff-NielsenBenthVeraart2012} and \citet{Podolskij2015} for surveys of relevant results.
		\qed
		\end{enumerate}
	\end{remark}

	\begin{remark}\label{remark/1/solution/nobrownian}
		Observe from \eqref{equation/1/preliminaries/drivingnoise/S} that the Brownian component of $L$ appears in $S$ as a positive drift $\sigma_L^2(t+r)$. In light of equations \eqref{equation/1/semimartingale/xi} and \eqref{equation/1/semimartingale/XSDDE}, we could absorb this drift into the constant $c_\mu$ and the function $f_\mu$, by replacing $c_\mu $ with $ c_\mu - \sigma_L^2 c_\nu$ and $f_\mu$ with $f_\mu + \sigma_L^2 f_\nu$ and changing the region of integration accordingly. Therefore from here onwards, without any loss of generality, we will assume $\sigma_L^2=0$ so that $S$ is a pure jump L\'evy process.
	\end{remark}

\subsection{Path properties of the solution}\label{section/1/solution/pathproperties}
	In Section \ref{section/1/solution/compoundpoisson} we show that the solution $X$ to \eqref{equation/1/semimartingale/XSDDE} driven by a general $S$, possibly with infinite activity, can be approximated by solutions $X^n$ to \eqref{equation/1/semimartingale/XSDDE} driven by compound Poisson processes $(S^n)_n$ that approximate $S$.
	We thus formulate some path properties of $X$ in the case when $L$ (and hence $S$) is a compound Poisson process, and compare them to the COGARCH process. 
	\begin{proposition}\label{proposition/1/semimartingale/pathofX}
		Let $(L_t)_{t\ge 0}$ be a compound Poisson process, i.e.  $\sigma_L =0$ and the L\'evy measure $\Pi_L$ of $L$ has finite total mass. Let $-r< T_0<T_1<\ldots$ and $\Delta L_t$ be the jump times and sizes of $L$.
		\begin{enumerate}[(a)]
			\item \label{proposition/1/semimartingale/pathofX/jumps}
			The jumps of the CDGARCH$(p,q)$ variance process $X$ are driven by the jumps of the quadratic variation process $S=[L,L]$, i.e., 
			\begin{align*}
				\Delta X_t = c_\nu X_{t-}\Delta S_t = c_\nu X_{t-}(\Delta L_t)^2, \quad t\ge -r.
			\end{align*}

			\item \label{proposition/1/semimartingale/pathofX/betweenjumps}
			If $f_\nu(-q)=0$, then on each $[T_n, T_{n+1})$, $\xi(X)$ is continuous and $X$ is continuously differentiable, with derivative given by
			\begin{align*}
				\frac{d}{dt} X_t = \eta -c_\mu X_t
				+\xi(X)_t, \quad t\in (T_j, T_{j+1}).
			\end{align*}
			Furthermore, when $p>0$ and $q=0$, between two consecutive jump times, the CDGARCH$(p,0)$ variance process $X$ follows the deterministic delay differential equation
			\begin{align*}
				\frac{d}{dt}X_t = \eta -c_\mu X_t + \int_{-p}^0 f_\mu(u) X_{t+u} du, 
				\quad t\in (T_j, T_{j+1}).\numberthis\label{equation/1/semimartingale/pathofX/CObetweenjumps}
			\end{align*}
			In the case $p=q=0$, i.e. when $X$ is a COGARCH process, $X$ decays exponentially between its jump times, and we have a closed form solution 
			\begin{equation*}
				X_t = \frac{\eta}{c_\mu} + \left(X_{T_{j}+} - \frac{\eta}{c_\mu}\right)e^{-c_\mu(t-T_j)},
				\quad t\in (T_j, T_{j+1}).  \qed
			\end{equation*}
		\end{enumerate}
	\end{proposition}

	It is clear that the jump structure of the variance process $X$ is the same for the COGARCH process and the COGARCH$(p,q)$ process. However, the behavior of $X$ in between jumps is very different. 
	In fact, from Proposition \ref{proposition/1/semimartingale/pathofX}, we immediately see two levels of generalization from the COGARCH process, which decays exponentially between jumps. We illustrate this graphically by simulating sample paths of the different CDGARCH processes via a simple Euler scheme.

	\begin{figure}[H]
		\centering\vspace{-4mm}\hspace{-0mm}\includegraphics[width=\textwidth]{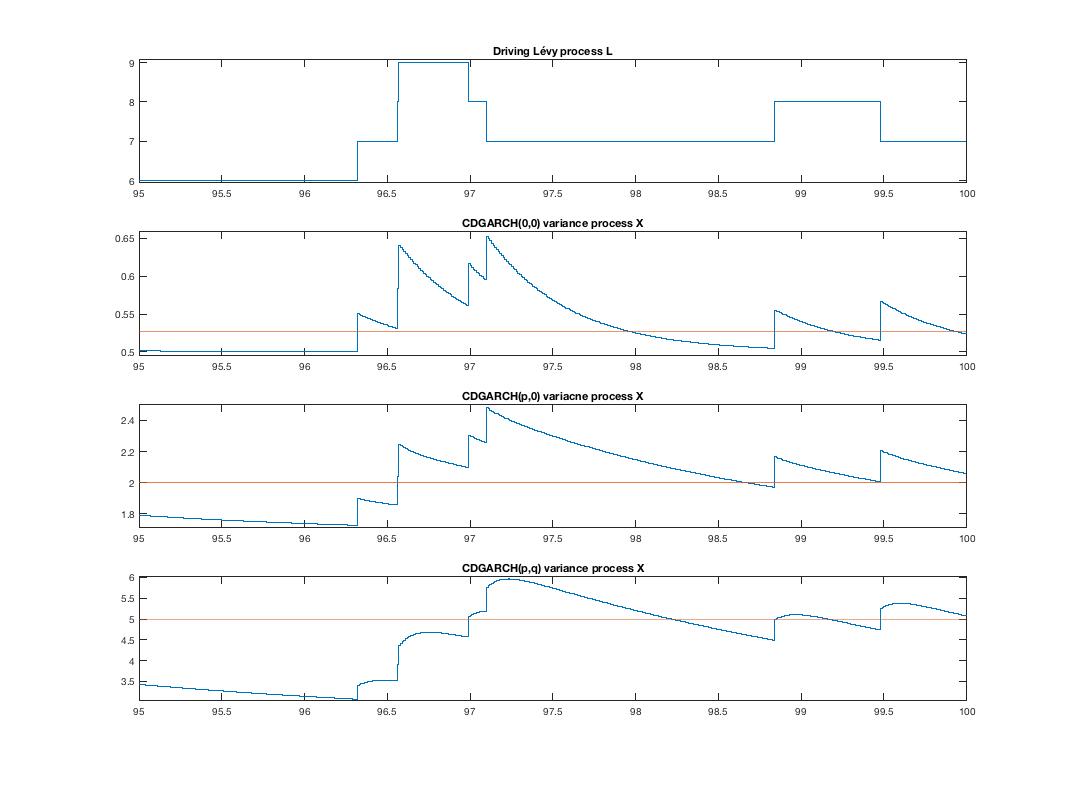}\vspace{-10mm}\caption{Driving noise $L$, CDGARCH(0,0), CDGARCH(p,0) and CDGARCH(p,q) variance processes}\label{fig/1/semimartingale/propogation}
	\end{figure}

	The top figure in Figure \ref{fig/1/semimartingale/propogation} is the simulated path of a compound Poisson process $L$ with unit intensity and jumps equal to $\pm 1$ with equal probability. The processes below are the COGARCH (or CDGARCH(0,0)), the CDGARCH(p,0) and the CDGARCH(p,q) variance processes, driven by the same realization of $L$.  The horizontal lines are the theoretical (stationary) means of the variance processes, computed in Section \ref{section/1/moments}. 
	The delay functions $f_\mu$ and $f_\nu$ are chosen to be exponential, and comparable parameters are chosen between all three processes. 

	In the CDGARCH$(p,0)$ case, the  process $X$ follows a deterministic differential equation (given in Proposition \ref{proposition/1/semimartingale/pathofX}) between the jumps of the driving noise $L$, but the decay towards the baseline level is slower than the exponential function, indicating a longer memory effect. 
	In the case $p,q>0$, $X$ is no longer deterministic between jumps, but is given by a continuous process of finite variation that depends on $\{X_u, u\in [T_j-p, T_j]\}$ as well as $\{\Delta S(u), u\in [T_j-q, T_j]\}$. For this particular simulated process, $X$ increases immediately after a jump, then starts decaying towards the baseline level. Depending on choices and sizes of  $f_\nu$, it is possible to have very different behaviors between jumps.

	We now give conditions for the positivity of $X$ when $S$ is a compound Poisson process, which we extend to the general case in the next section.
	\begin{proposition}\label{proposition/1/solution/CPPpositivity}
		Let $S$ be a compound Poisson process satisfying Assumption \ref{assumptions/1/existence}(\ref{assumptions/1/existence/S}) and $X$ be the unique solution to \eqref{equation/1/semimartingale/XSDDE} driven by $S$. Suppose $\eta>0$, $c_\mu >\norm{f_\mu}_{L^1}$ and 
		\begin{align*}
		 	X_u\ge x^- := \frac{\eta}{c_\mu - \norm{f_\mu}_{L^1}}>0,\numberthis\label{equation/1/semimartingale/lowerbound}
		 	\quad\forall u\in[-r,0].		
		\end{align*} 
		Then $X_t\ge x^-$ for all $t>0$, i.e. $X$ is positive and bounded away from zero.\qed
	\end{proposition}

\subsection{Approximation}\label{section/1/solution/compoundpoisson}
	
	Following Remark \ref{remark/1/solution/nobrownian}, we will assume that $S$ takes the form
	\begin{align*}
		S_t = \sum_{-r<s\le t} (\Delta L_s)^2.
		\numberthis\label{equation/1/solution/S without BM}
	\end{align*}
	For each $n\in \N$, define the approximating process $S^n$ by
	\begin{align*}
		S^n_t:= \sum_{-r<s\le t} (\Delta L_s)^2 \ind_{ \{|\Delta L_s| \ge \frac{1}{n}\} }.
		\numberthis\label{equation/1/solution/Sn without BM}
	\end{align*}
	Then $(S^n)_n$ is a sequence of compound Poisson processes satisfying $S^n_t \le S_t$ for all $n\in\N$ and $t\ge -r$. 
	For each $n$, we will consider equation \eqref{equation/1/semimartingale/XSDDE} driven by $S^n$:
	\begin{align*}
		dX^n_t = b^n(X^n)_t dt + c_\nu X^n_{t-} dS_t^n, 
		\numberthis\label{equation/1/solution/approximation/sdden}
	\end{align*}
	where the drift coefficient $b^n:\Rm_+\times \Rm\times \Omega$ is defined as
	\begin{align*}
		b^n(H)_t :=  \eta - c_\mu H_t 
		+ \int_{t-p}^t f_\mu (u-t) H_u du 
		+ \int_{t-q+}^t f_\nu (u-t) H_{u-} dS_u^n.
	\end{align*}
	By Theorem \ref{theorem/1/existence}, Equation \eqref{equation/1/solution/approximation/sdden} has a unique solution $X^n$ in $\cS^2$ for each initial value $\Phi$ satisfying Assumption \ref{assumptions/1/existence}(\ref{assumptions/1/existence/Phi}). Similar to $b^n$, we will write $b$ for the drift coefficient of \eqref{equation/1/semimartingale/XSDDE} driven by $S$. 

	Recall from \cite{Protter2004} that a sequence of processes $(H^n)_n$ converges to $H$ uniformly on compacts in probability (ucp) if for each $t>0$, $\sup_{s\le t} |H_s^n - H_s|$ converges to 0 in probability.
	To show $X^n$ approximates $X$ in the ucp topology, we will need the following set of results, which are interesting in their own right. Let $(U_t)_{t\ge 0}$ be the (finite and increasing) process given by
	\begin{align*}
		U_t :=c_\mu + \norm{f_\mu}_{L^1} + (\norm{f_\nu}_{L^1} + f_\nu(0))S_t^*.
		\numberthis\label{equation/1/approximation/lipschitz process}
	\end{align*}

	\begin{proposition}\label{proposition/1/solution/approximation}
		Let $X$ and $(X^n)_{n\in\N}$ be the unique solutions to \eqref{equation/1/semimartingale/XSDDE} and \eqref{equation/1/solution/approximation/sdden}.

		\begin{enumerate}[(a)]
			\item \label{proposition/1/solution/approximation/Sn to S} 
			For each $t\ge 0$, $S^n$ converges to $S$ in $\norm{\cdot}_t$ and hence in ucp.
			
			\item \label{proposition/1/solution/approximation/functional lipschitz}
			The drift coefficient $b$ is functional Lipschitz (page 256, \cite{Protter2004}) with Lipschitz process $(U_t)_t$ defined in \eqref{equation/1/approximation/lipschitz process}. That is, for every $Y$ and $Z$ in $\cS^2$, we have
			\begin{align*}
			|b(Y)_t - b(Z)_t| \le  U_t \sup_{s\le t} |Y_s - Z_s|, \quad \forall t\ge 0. 	
			\end{align*} 
			Furthermore, for every $n\in\N$, the coefficient $b^n$ is functional Lipschitz with the same $U$. 

			\item \label{proposition/1/solution/approximation/bn to b}
			For each $t\ge 0$, the process $b^n(X)$ converges to $b(X)$ in $\norm{\cdot}_t$ and hence in ucp.
			\qed
		\end{enumerate}
		
	\end{proposition}

	We are now in a position to state the approximation result, which allows us to extend Proposition \ref{proposition/1/solution/CPPpositivity} to a general driving noise $S$, thus showing positivity of the CDGARCH variance process $X$. 
	\begin{theorem}\label{theorem/1/solution/approximation}
		Let $S$ and $(S^n)_n$ be given by \eqref{equation/1/solution/S without BM} and \eqref{equation/1/solution/Sn without BM}. Let $X$ and $(X^n)_n$ be the corresponding solutions of \eqref{equation/1/semimartingale/XSDDE} and \eqref{equation/1/solution/approximation/sdden}. 
		Then as $n\to \infty$, $X^n$ converges to $X$ in ucp.
		\qed
	\end{theorem}
	
	\begin{theorem}\label{theorem/1/solution/positivity}
		Proposition \ref{proposition/1/solution/CPPpositivity} holds for any $S$ of the form \eqref{equation/1/solution/S without BM} satisfying Assumption \ref{assumptions/1/existence}(\ref{assumptions/1/existence/S}). In particular, suppose $\eta>0$ and $c_\mu >\norm{f_\mu}_{L^1}$, then for each $t>0$, $X_t$ is positive and bounded away from zero by $x^-$ defined in \eqref{equation/1/semimartingale/lowerbound}. 
		\qed
	\end{theorem}

We conclude the current section with simulated paths of the CDGARCH(p,q) price and variance processes.
Figure \ref{fig/1/semimartingale/paths} shows simulated sample paths of the driving noise $L$, the return process $dY$, the CDGARCH process $Y$ and its volatility process $\sqrt{X}$. Note the return process visibly exhibits volatility clustering which is reflected in the process $\sqrt{X}$.
\begin{figure}[H]
		\centering\vspace{-4mm}\hspace{-0mm}\includegraphics[width=\textwidth]{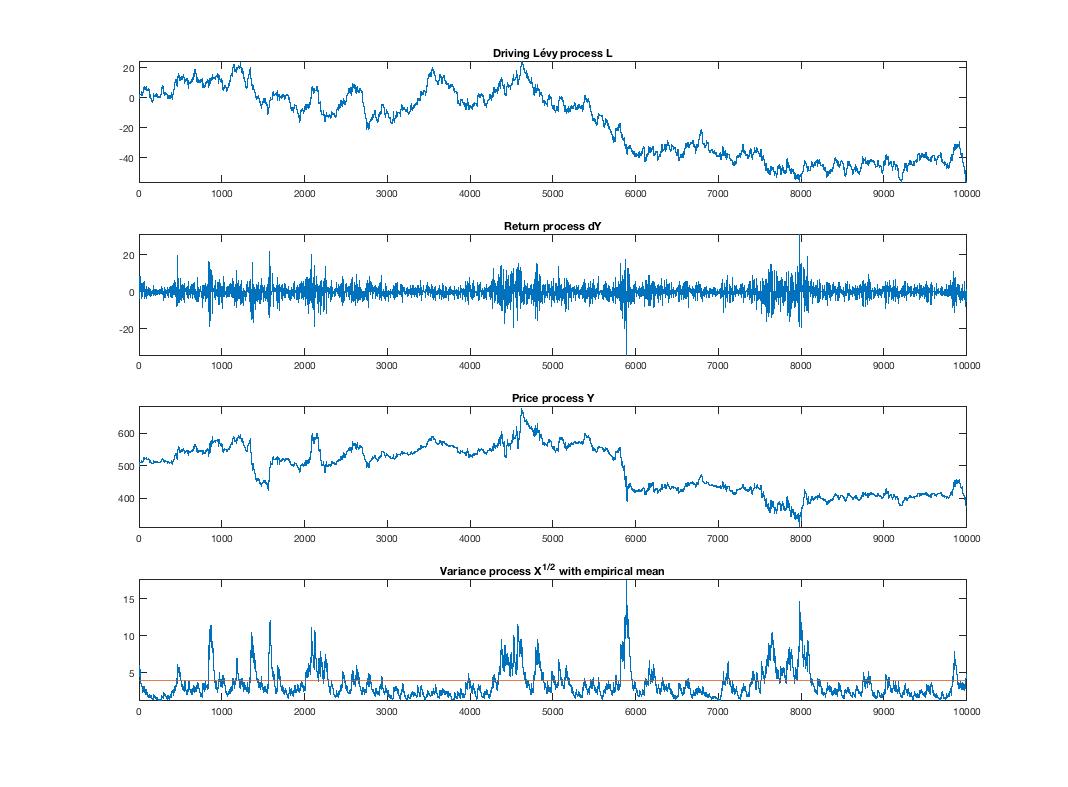}\vspace{-10mm}\caption{Simulated paths of the processes $L$, $dY$, $Y$ and $X$}\label{fig/1/semimartingale/paths}
	\end{figure}

\section{Moments}\label{section/1/moments}

\subsection{Uniform moment bounds}\label{section/1/moments/uniform bounds}
We first provide uniform $L^1$ and $L^2$ bounds on $X$. Let $\eta$ be defined as in \eqref{equation/1/model/CDGARCH_X} and the constants $C_1^+, C_1^-, C_2^+, C_2^-$ be given by
	\begin{align*}
		C_1^{\pm} & =  c_\mu - \kappa_2 c_\nu  \pm \Big(\norm{f_\mu}_{L^1} + \kappa_2 \norm{f_\nu}_{L^1}\Big),
		   \\
		C_2^{\pm} & =  c_\mu - \kappa_2 c_\nu - \frac{1}{2}\kappa_4 c_\nu^2  \pm \Big(\norm{f_\mu}_{L^1} + \kappa_4 \norm{f_\nu}_{L^2}\Big).
	\end{align*}

\begin{theorem}\label{theorem/1/moments/moment bound}
	Suppose Assumptions \ref{assumptions/1/existence} hold and let $(X_t)_{t\ge 0}$ be a positive solution to \eqref{equation/1/semimartingale/XSDDE}. 
	\begin{enumerate}[(a)]
		\item Suppose $\Em[X_0]<\infty$ and $C_1^->0$, or equivalently,\label{theorem/1/moments/moment bound/L1uniform}
		\begin{align*}
			c_\mu > \kappa_2 c_\nu  + \norm{f_\mu}_{L^1} + \kappa_2 \norm{f_\nu}_{L^1}.
			\numberthis\label{condition/1/moments/L1uniform}
		\end{align*}
		Then $X$ is uniformly bounded in $L^1$ with 
		$\sup_{t}\Em[X_t]  \le {2 \eta}/{C_1^-}+ \Em[X_0]{C_1^+}/{C_1^-}$.

		\item Suppose $\Em[X_0^2]<\infty$ and $C_2^->0$, or equivalently,\label{theorem/1/moments/moment bound/L2uniform}
		\begin{align*}
			c_\mu > \kappa_2 c_\nu + \frac{1}{2}\kappa_4 c_\nu^2  + \norm{f_\mu}_{L^1} + \kappa_4 \norm{f_\nu}_{L^2}.
		 	\numberthis\label{condition/1/moments/L1uniform}
		\end{align*} 
		Then $X$ is uniformly bounded in $L^2$ with
		$\sup_t \Em[X_t^2]  \le \left({\eta}/{C_2^-}\right)^2 + \Em[X_0^2]{C_2^+}/{C_2^-}$.\qed
		
	\end{enumerate}
\end{theorem}

\subsection{Moment processes}\label{section/1/moments/mean function}
Let $m$ be the mean function of $X$, i.e., $m(t):= \Em [X_t]$, $t\in[-r,\infty)$.
Write $\phi(\cdot)$ for the mean function of the initial segment $\Phi$. 
For $t>0$, define the \textit{segment process} $m_{(t)}:[-r,0]\to\R$ of the process $m$ as $m_{(t)}(u):= m(t+u),  u\in [-r,0]$. For notational simplicity, we will from here onwards write $c_0 := c_\mu  - \kappa_2 c_\nu$ and $f:= f_\mu + \kappa_2 f_\nu$. 

	\begin{theorem}\label{theorem/1/moments/mean}
		If Assumptions \ref{assumptions/1/existence} are satisfied, then for any positive $\phi\in\scD$, 
		\begin{enumerate}[(a)]
			\item \label{theorem/1/moments/meanFDE} The mean function $m:[-r,\infty)\to\Rm$ is finite-valued,
			continuously differentiable on $(0,\infty)$, and satisfies the (deterministic) functional differential equation
			\begin{align*}
					\frac{d}{dt} m(t)
					& =  \eta - c_0 m(t) + \int_{-r}^0 m(t+u) f(u) du,\quad t>0, 
					\numberthis\label{equation/1/moments/meanFDE}	
			\end{align*}
			with the initial condition $m(u) = \phi(u)$ for $u\in[-r,0]$, where $\phi\in \scD$. 


			\item \label{theorem/1/moments/meanRE} The mean function $m$ also satisfies the renewal equation
			\begin{align*}
				m(t) = \int_0^t \zeta(t-u) m(u) du + h(t),\quad t>0, 
			\end{align*}
			with initial condition $m(0)= \phi(0)$, with convolution kernel $\zeta$ given by 
			\begin{align*}
			 	\zeta(t)  = - c_0 \ind_{(0,\infty)}(t)+ \int_0^{t\wedge r} f(-u) du, \quad t\in[0,\infty),
			 	\numberthis\label{equation/1/moments/meanNBV}
			\end{align*}
			and forcing function $h:[0,\infty)\to \R$  given by
			\begin{align*}
				h(t) =m_{(0)}(0)+ \frac{\eta}{\zeta(r)}\int_0^t \zeta(u)du
				+\int_0^r (\zeta(t+u) - \zeta(u) )\left(m_{(0)}(-u) +  \frac{\eta}{\zeta(r)}\right) du.\qed
			\end{align*}

		\end{enumerate}
	\end{theorem}	
	
	From Theorem \ref{theorem/1/moments/mean}, we can formulate necessary and sufficient conditions for the process $X$ to be mean stationary, i.e., for $m(t) = M$ for all $t\in [0,\infty)$ for some $M\ge 0$, or for $X$ to be asymptotically mean stationary, i.e., for $m(t)\to M$ when $t\to \infty$. 

	\begin{theorem}\label{theorem/1/moments/stationarymean}
		Assume Theorem \ref{theorem/1/moments/mean} holds so $m(t)$ satisfies the functional differential equation \eqref{equation/1/moments/meanFDE} with some positive initial condition $\phi\in\scD$. Then 
		\begin{enumerate}[(a)]
			\item \label{theorem/1/moments/stationarymean/limitingmean} 
			The mean function $m$ converges to a (positive) limit $M$ exponentially fast as $t\to \infty$, if and only if $c_0>\norm{f}_{1}$. If it exists, the limit $M$ is uniquely given by
			\begin{align}
				M  =\frac{\eta}{c_0- \norm{f}_{1}}.
				\label{equation/1/moments/Mx}
			\end{align}

			\item \label{theorem/1/moments/stationarymean/stationarymean}
			The process $X$ admits a stationary (positive) mean, i.e. $m(t)=M$ for all $t\in[0,\infty)$, if and only if $c_0>\norm{f}_{L^1}$ and $\phi\equiv M$ on $[-r,0]$, where $M$ is given by \eqref{equation/1/moments/Mx}.\qed
		\end{enumerate}

	\end{theorem}

	\begin{remark}
		Here, $c_\mu$ is negative and provides a mean reversion effect, while $c_\nu$ is positive representing a positive jump in $X$ whenever there is a jump in $S$ (c.f. Proposition \ref{proposition/1/semimartingale/pathofX}(\ref{proposition/1/semimartingale/pathofX/jumps})). The functions $f_\mu$ and $f_\nu$ are positive functions representing the delayed effects of past values of $X$ and $S$.
		The key condition for first order stability, $c_0>\norm{f}_{1}$, or
		\begin{align*}
			\numberthis\label{equation/1/moments/stationarymean/condition}
			c_\mu > \kappa_2 c_\nu + \int_{-r}^0 (f_\mu + \kappa_2 f_\nu )(u) du,
		\end{align*} 
		then has a very natural interpretation: the speed of mean reversion has to be large enough in comparison to the delay effects. \qed
	\end{remark}

	In general, the second moment of the process $X$ involves the term
	$\Em[X_t\int_{t-q+}^t X_s f_\nu(s) dS_s]$, which is not easily evaluated. 
	However, we can still formulate some asymptotic results.
	\begin{theorem}\label{theorem/1/moments/weakdependence}
	Suppose condition \eqref{equation/1/moments/stationarymean/condition} is satisfied so that $\Em[X_u] \to M$ as $u\to \infty$. For every $t>0$ and $\cF_t$-measurable random variable $Z$ with $\Em[Z^2 X_u^2]<\infty$, we have $\Em[Z X_u]\to M\Em[Z]$
		exponentially fast as $u\to \infty$. 
	\qed
	\end{theorem}
	
	The asymptotic behavior of the covariance function $\Cov(X_t, X_{t+u})$ of the process $X$ is an immediate corollary to Theorem \ref{theorem/1/moments/weakdependence} by taking $Z = X_t$. 
	\begin{corollary}\label{corollary/1/moments/weakdependence}
		Suppose $X$ is asymptotically mean stationary and has finite fourth moments. Then for every $t>0$, the covariance function $\Cov(X_t, X_{t+u})$ tends to zero exponentially fast as $u\to \infty$. 
		Thus, the process $X$ possesses ``short memory''. \qed
	\end{corollary}

We finally look at the properties of the price and return processes under the CDGARCH model.
Recall the price process $Y_t = Y_0 + \int_{0+}^t \rt{X_{s-}} dL_s, t\ge 0$, 
and define 
the return process $(\td Y_t)_{t> 1}$ by $\td Y_t:= Y_{t} - Y_{t-1} = \int_{t-1+}^t \rt{X_{s-}}dL_s$. 
\begin{theorem}
Let $(X_t)_{t}$ be the solution to \eqref{equation/1/semimartingale/XSDDE} and $Y, \td Y$ be defined as above. 
Suppose $X$ is mean stationary, with mean $M$ defined in \eqref{equation/1/moments/Mx}.
	\label{theorem/1/moments/price}
	\begin{enumerate}[(a)]

		\item \label{theorem/1/moments/price/statioanrycovariance}
		The return process $\td Y$ is covariance stationary, with zero mean and auto-covariance function given by
		$\Cov(\td Y_t, \td Y_{t+u}) = \Em\left[\td Y_t \td Y_{t+u}\right] = \kappa_2 M (1-u)_+$.

		\item \label{theorem/1/moments/price/squaredreturns} Suppose $\td Y$ has finite fourth moments. Then for any $t>1$, the squared return process $(\td Y_t^2)_{t> 1}$ satisfies $\Cov(\td Y_t^2, \td Y_{t+u}^2)\to 0$ exponentially fast as $u\to \infty$, i.e., $Z^2$ has short memory.	\qed


	\end{enumerate}
\end{theorem}



\section{Proofs}\label{section/1/proofs}
\subsection{Existence and uniqueness of the solution}\label{section/proofs/1/existence}

Given a signed measure $\mu$ on a measure space $(S, \Sigma)$, we denote its corresponding total variation measure $|\mu|$ by $|\mu|(E)= \sup_{\pi} \sum_{A\in \pi} |\mu(A)|$, for all $E\in \Sigma$, where the supremum is taken over all $\Sigma$-measurable partitions $\pi$ of $E$. We also denote the total variation norm of $\mu$ as $\norm{\mu}:=|\mu|(S)$.

Let $\mu$ and $\nu$ be signed Borel measures on $[-r,0]$ with finite total variations and $S$ be a c\`adl\`ag, adapted process with paths of finite variation. Recall $\theta$ from \eqref{equation/1/model/theta} and $\norm{\cdot}_t$ from \eqref{equation/1/existence/normt}. We can write the variance equation \eqref{equation/1/model/CDGARCH_X} as $X_t= \theta_t + \cR (X)_t$, where $\cR$ is a linear map on $\scD$ given by 
	\begin{align}
	 	\cR(Z)_t= 
	 		\int_{-r}^0  \int_{u}^{t+u} Z_s ds  \mu (du) 
			+ \int_{-r}^0  \int_{u+}^{t+u} Z_{s-} dS_s \nu (du),\quad t>0,
		\label{equation/1/existence/proof/existence/cR}
	\end{align}
and $\cR(Z)_u=0$ for all $u\le 0$. Using Lemma \ref{lemma/1/preliminaries/Jacod}, it is easy to see that $\cR(Z)$ is c\`adl\`ag and adapted, whenever $Z$ is c\`adl\`ag and adapted. 
We first obtain some norm estimates on $\cR$:
	\begin{lemma}\label{lemma/1/existence/norm} 
		Let $(H_t)_{t\ge -r}$ be a process in $\cS^2$, as defined in \eqref{equation/1/existence/normt}.
		Then for all  $T\ge -r$,
		\begin{align*}
			\norm{\cR(H)}_{T}^2
				\le K_T \int_{-r}^T \Em\left[H_{s}^2 \right]ds\le K'_T \norm{H}_{T}^2,
		\end{align*}
		where $K_T=2(\norm{\mu}^2+ 2\kappa_2 \norm{\nu}^2)T + 16 \kappa_4 \norm{\nu}^2<\infty$ and
		$K'_T=K_T(T+{r})<\infty$.

		\begin{proof}
			For notational convenience, we will define the semi-norm
			\begin{align*}
				\norm{Z}_{[0,t]}:
				=\left( \Em\left[\sup_{s\in[0,t]} |Z_s|^2\right]\right)^{1/2}.
			\end{align*}
			Since $\cR(H)_u=0$ for $u\le 0$, by the inequality $(a+b)^2\le 2a^2 + 2b^2$, 
			we have 
			\begin{align*}
				\norm{\cR (H)}_{T}^2
				&\le  2\normm{ 
						\int_{-r}^0  \int_{u}^{u+\bigcdot}  H_s  ds  \mu (du) }_{[0,T]}^2
					+ 2\normm{ 
						\int_{-r}^0  \int_{u+}^{u+\bigcdot}  H_{s-}  dS_s \nu (du)}_{[0,T]}^2
				 =: \mathrm{\bf I} + \mathrm{\bf II}.
			\end{align*}
			Since $H=0$ on $[-r,0]$, an application of the Cauchy-Schwarz inequality yields the bound
			\begin{align*}
				\mathrm{\textbf I}
				&\hspace{0cm}\le 2\norm{\mu}^2
					\Em \left[ \sup_{t\in[0,T]}\sup_{u\in[-r,0]}  
					\left(t \int_{u}^{t+u}\left| H_s\right|^2 ds\right)  \right]
				\le 2T\norm{\mu}^2  \Em \left[ \int_{0}^T |H_s|^2 ds \right].
			\end{align*}
			For $\mathrm{\bf II}$, recall $\td S_t:= S_t- (t+r)\kappa_2$ from \eqref{equation/1/preliminaries/drivingnoise/tdS}. 
			Using the same reasoning as above,  
			\begin{align*}
				\mathrm{\bf II}&\le 4\kappa_2^2 \normm{
					\int_{-r}^0 \int_{u}^{u+\bigcdot}  H_{s}  ds \nu (du) }_{[0,T]}^2
				+4  \normm{ \int_{-r}^0 \int_{u+}^{u+\bigcdot}  H_{s-}  d\td S_s \nu (du) }_{[0,T]}^2
					=:\mathrm{\bf III}+ \mathrm{\bf IV}.
			\end{align*}
			By similar workings as in $\mathrm{\bf I}$, we have $\mathrm{\bf III}
				\le 
				4\kappa_2^2T\norm{\nu}^2 \Em \left[ \int_{0}^T |H_s|^2 ds \right].$
			For $\mathrm{\bf IV}$, 
			recall $\td S$ is a square integrable martingale and $d\<\td S \>_t = \kappa_4 dt$. 
			Since $\norm{H}_T<\infty$, $H$ is clearly in $L^2(\td S)$, and the process $H\cdot \td S$ is a square integrable martingale by Lemma \ref{lemma/1/preliminaries/Jacod}. By Jensen's inequality, Doob's inequality and the Ito isometry, we have
			\begin{align*}
				\mathrm{\bf IV}
				&\hspace{0cm}
				\le  4\norm{\nu} \int_{-r}^0 \Em \left[\sup_{t\in[0,T]} \left|\int_{u+}^{t+u}  H_{s-}  d\td S_s\right|^2\right] |\nu| (du)
				\le  16\norm{\nu} \int_{-r}^0 \Em \left[\left|\int_{u+}^{T+u}  H_{s-}  d\td S_s\right|^2\right] |\nu| (du)
				\\
				&\hspace{0cm}
				\le 16 \norm{\nu}^2 \sup_{u\in[-r,0]} \Em \left[\int_{u}^{T+u}  H_{s}^2  d \langle\td S,\td S\rangle_s\right] 
				\le 16 \kappa_4\norm{\nu}^2 \Em \left[\int_{-r}^{T}  H_{s}^2  ds\right]
			\end{align*}
			The lemma follows immediately by collecting all terms. 
		\end{proof} 
	\end{lemma} 

	\begin{proof}[Proof of Theorem \ref{theorem/1/existence}]
	\pf{(\ref{theorem/1/existence/1}).}
		Let $S$ and $\Phi$ satisfy Assumptions \ref{assumptions/1/existence} and $\theta$ be defined in \eqref{equation/1/model/theta}. For the existence of a solution, we use a  Picard iteration to produce a sequence of $\cS^2$-processes that converges to a limit. 
		Set the initial term $X^{(0)}:= \theta\in \cS^2$, and define recursively for each $n\ge  1$ the process $X^{(n)}:= \theta+\cR X^{(n-1)}$. We see that the differences between each term are given by $X^{(1)}-X^{(0)}=\cR(\theta)$ and $X^{(n)}- X^{(n-1)}= \cR(X^{(n-1)}- X^{(n-2)})$, for $n\ge 2$. 
		
		Write $D_{n,T}:=\norm{X^{(n+1)}-X^{(n)}}_T$, so that $D_{0,T}=  \norm{\cR(\theta)}_T$  and $D_{n,T} = \norm{\cR\left(X^{(n)}-X^{(n-1)}\right)}_T$, for $n\ge 1$. 
		The first term $D_0$ is finite by an application of Lemma \ref{lemma/1/existence/norm} to $H=\theta$. Since $X^{(n+1)} = \theta + \cR(X^{(n)})$, by Lemma \ref{lemma/1/preliminaries/Jacod} and the second bound in Lemma \ref{lemma/1/existence/norm}, $X^{(n+1)}$ is in $\cS^2$  whenever $X^{(n)}$ is in $\cS^2$. Therefore by induction, for each $n\ge 1$, the difference $X^{(n)}-X^{(n-1)}$ is in $\cS^2$ and we can apply Lemma \ref{lemma/1/existence/norm} to each $D_{n,T}$. 

		Since $t\mapsto D_{n,t}$ is non-decreasing and non-negative on $[-r,T]$, applying Lemma \ref{lemma/1/existence/norm} to each $D_{n,T}$ and expanding the recursion yields a Gronwall type inequality
		\begin{align*}
			D_{n,T}^2 &\le K_T  \int_{0}^{T} D_{n-1,t_2}^2 dt_2
			\le K_T ^{n} \int_{0}^{T} \int_{0}^{t_2}\cdots \int_{0}^{t_{n}} 
				D_{0,t_{n+1}}^2 d{t_{n+1}}\cdots d{t_3} d{t_2}
			\le \frac{K_T ^{n} T^n}{n!} D_{0,T}^2.
		\end{align*}
		The sequence $(D_{n,T})_n$ is therefore Cauchy for each $T>0$. Since $\scD[-r,\infty)$ is complete in $\norm{\cdot}_T$, taking $n\to \infty$, the sequence of processes 
		\begin{align*}
			X^{(n)}=X^{(0)}+\sum_{k=1}^n \left(X^{(k)}-X^{(k-1)}\right)
		\end{align*}
		converges in $\norm{\cdot}_T$ to a limit $X$, which is also in $\cS^2$. 

		It remains to show that this limit $X$ is indeed a solution to \eqref{equation/1/model/CDGARCH_X}, i.e. satisfies $X=\theta+\cR(X)$. First, observe that since $(D_{n,T})_n$ is summable for every $T>0$, 
		\begin{align}
			\norm{X-X^{(n)}}_T
			&=\normm{\sum_{k=n}^\infty X^{(k+1)}-X^{(k)}}_T\le \sum_{k=n}^\infty D^{(k)}(T)\to0,
			\label{equation/1/existence/proof/Xn_norm}
		\end{align}
		as $n\to\infty$. Using $X^{(n)} = \theta + \cR(X^{(n-1)})$, by the second bound from Lemma \ref{lemma/1/existence/norm}, we also have
		\begin{align*}
		 	\norm{\theta+\cR X-X^{(n)}}^2_T=\norm{\cR (X-X^{(n-1)})}^2_T\le K'_T \norm{X-X^{(n-1)}}^2_T\to0.
		\end{align*}
		Then by the triangle inequality, for any $n\ge 1$, 
		\begin{align*}
			\norm{\theta+\cR X-X}_T\le \norm{\theta+\cR X- X^{(n)} }_T+\norm{X- X^{(n)} }_T,
		\end{align*}
		which implies that $\sup_{0\le s\le  T}|\theta_s+\cR X_s-X_s|=0$ a.s. and  hence  \eqref{equation/1/model/CDGARCH_X} is satisfied on any $[0,T]$. 

		To establish uniqueness, suppose $X$ and $X'$ are two strong solutions to \eqref{equation/1/model/CDGARCH_X}, i.e., we have $X=\theta+\cR(X)$ and $X'=\theta+\cR(X')$. 
		Let $D_T:=\norm{X-X'}_T=\norm{\cR(X-X')}_T$.
		By Lemma \ref{lemma/1/existence/norm}, for any $0\le T< \infty$,  we have
		$D_T^2\le K_T \int_{-r}^T D_t^2 dt$,
		where $K_T$ is defined in Lemma \ref{lemma/1/existence/norm}. By Gronwall's inequality (Thm V.68, \cite{Protter2004}), $\norm{X-X'}^2_T=0$ for all $t\ge 0$  which implies that $\sup_{-r\le s\le T}\left|X_s-X'_s\right|=0$, almost surely, and the two solutions are indistinguishable. 

\vspace{2mm}
	\pf{(\ref{theorem/1/existence/2}).}
		Since $|x|^\alpha\le  1+ |x|^2$ for any $0\le \alpha\le 2$ and $x\in\Rm$, we have, 
		\begin{align*}
			\Em \left[(X_T^*)^\alpha \right]
			\le 1 + \Em \left[(X_T^*)^2\right] <\infty.
		\end{align*}
		Hence the function $t\mapsto \Em [|X_t|^\alpha]$ is finite-valued. Since $X$ is c\`adl\`ag, by the dominated  convergence theorem with $|X^*_T|^\alpha$ as dominating functions, $t\mapsto \Em [|X_t|^\alpha]$ is a c\`adl\`ag function on $[0,\infty)$ for $0\le \alpha\le 2$. 
	\end{proof}

\subsection{Properties of the solution}\label{section/proofs/1/semimartingale}
\begin{proof}[Proof of Proposition \ref{proposition/1/semimartingale}]
\pf{(\ref{proposition/1/semimartingale/kernelF}).}
	We first rewrite $F_\nu(t,s) = \int_{-q}^0 \ind_{[s-t,s]}(u)f_{\nu}(u)du$. For any $(t_2, s_2)$ and $(t_1,s_1)\in\Rm^2$, by the triangle inequality, $F_\nu$ is Lipschitz on $\Rm^2$ since
	\begin{align*}
		&|F_\nu(t_2,s_2) - F_\nu(t_1, s_1)| 
		 \le |F_\nu(t_2,s_2) - F_\nu(t_1, s_2)|  + |F_\nu(t_1,s_2) - F_\nu(t_1, s_1)| 
		   \\
		&\qquad \le \int_{-q}^0 \left|\ind_{[s_2-t_2, s_2]}(u) - \ind_{[s_2-t_1, s_2]}(u)\right| f_{\nu}(u) du
		+ \int_{-q}^0 \left|\ind_{[s_2-t_1, s_2]}(u) - \ind_{[s_1-t_1, s_1]}(u)\right| f_{\nu}(u)du
		   \\
		&\qquad \le \norm{f_\nu}
			\Big(|t_2-t_1| + 2|s_2-s_1|\Big) 
		\le 2 \norm{f_\nu} \Big|(t_2,s_2) - (t_1,s_1)\Big|,
	\end{align*}
	where $\norm{\cdot}$ is the sup-norm and $|\cdot|$ is the Euclidean distance on $\Rm^n$. Similarly for $F_\mu$.

\vspace{2mm}
\pf{(\ref{proposition/1/semimartingale/Xi}).} 
	Since $F_\mu$ and $F_\nu$ are identically zero whenever $s\ge t$ or $s\le -r$, we will omit the region of integration and write $\Xi(X)_t = \int F_\mu(t,s) X_{s-}   ds + \int F_\nu(t,s) X_{s-}   dS_s$.
	
	Since for almost every $\omega\in\Omega$, $t\mapsto S_{t}(\omega)$ is a non-decreasing c\`adl\`ag function, we will fix such an  $\omega$ and treat the stochastic integral above as a Lebesgue-Stieljes integral with respect to the function $t\mapsto S_{t}(\omega)$. 
	Since $F_\mu$ and $F_\nu$ vanishes for $s\notin(-r,t)$, by Proposition \ref{proposition/1/semimartingale}(\ref{proposition/1/semimartingale/kernelF}), for any $t_2, t_1\in\Rm_+$, 
	\begin{align*}
		\left| \Xi(X)_{t_2} - \Xi(X)_{t_1} \right| 
		& \le 
		\int\left|F_\mu(t_2,s)-F_\mu(t_1,s)\right| |X_{s-}| ds
		+
		\int\left|F_\nu(t_2,s)-F_\nu(t_1,s)\right| |X_{s-}| dS_s
		   \\
		&\le  2 \norm{f_\mu} \left(\int_{-p }^{t_2\maxx t_1}|X_s| ds\right)|t_2-t_1|
		+ 2 \norm{f_\nu} \left(\int_{-q+}^{t_2\maxx t_1}|X_{s-}| dS_s\right)|t_2-t_1|.
	\end{align*}
	It follows that $t\mapsto \Xi(X)_t$ is locally Lipschitz continuous almost surely, since with probability one $X$ is locally bounded and $S$ has finite variation on compacts.

	We first compute $d F_\nu(t,s)/dt$ - the case of $F_\mu$ is identical and omitted. 
	In the expression $F_\nu(t,s) = \int_{-q}^0 \ind_{[s-t,s]}(u)f_{\nu}(u)du$,
	the integrand clearly does not depend on $t$ whenever $t\notin(s\maxx0, s+q)$, 
	hence $t\mapsto F_\nu(t,s)$ is constant on these regions and $d F_\nu(t,s)/dt=0$. 
	On the interval $t\in (s\vee0, s+q)$, we have $F_\nu(t,s) = \int^{s\wedge0}_{s-t} f_\nu(u) du$ so by the Fundamental Theorem of Calculus, $t \mapsto F_\nu(t,s)$ is continuously differentiable and $d F_\nu(t,s)/dt = f_\nu(s-t)$ on the interval $t\in (s\vee0, s+q)$. 
	We can therefore write $d F_\nu(t,s)/dt = f_\nu(s-t)\ind_{[s\maxx0, s+q]}(t)$ for almost every $t\ge 0$. 
	Clearly, $t\mapsto F_\nu(t,s)$ is not differentiable at $t = s\maxx 0$ or $t = s+q$, unless $f_\nu(0)$ and $f_\nu(-q)$ are equal to zero. 

	We now compute the derivative of the second integral in \eqref{equation/1/semimartingale/Xi(X)VolterraKernel}:
	\begin{align*}
		I_t:=\int F_\nu(t,s) X_{s-} dS_s, t\ge 0. \numberthis\label{equation/1/semimartingale/proof/Iasintegral}
	\end{align*}
	The case of the first integral is similar and omitted. Again, we fix an $\omega\in \Omega$ such that $t\mapsto S_t$ is a non-decreasing c\`adl\`ag function and treat the $dS$ integral as a Stieljes integral. 
	For every $t\in\Rm_+$, the map $s\mapsto F_\nu(t,s)X_{s-}$ is in $L\loc^1(\Rm, dS)$ since $X$ and $S$ are locally bounded and $s\mapsto F_\nu(t,s)$ is supported on a compact set.
	For every $s\in\Rm$, the map $t\mapsto F_\nu(t,s)X_{s-}$ is continuously differentiable in $(s\maxx0, s+q)$ by the previous argument. For every $t$, the derivative $s\mapsto \frac{d}{dt} F_\nu(t,s)X_{s-}$ is locally bounded and hence also in $L\loc^1(\Rm, dS)$. Then by the differentiation lemma (\cite[Theorem 6.28]{Klenke2013}), $t\mapsto I_t$ is differentiable almost everywhere with derivative
	\begin{align*}
		\frac{d}{dt} I_t & 
		= \int f_\nu(u-t) \ind_{[u\vee0, u+q]}(t)  X_{u-} dS_u
		= \int_{(t-q, t]} f_\nu(u-t) X_{u-} dS_u,\quad t>0.
		\numberthis\label{equation/1/semimartingale/proof/Iasintegralderivative}
	\end{align*}
	The expression \eqref{equation/1/semimartingale/xi} then follows with a simple change of variable.
\end{proof}

\begin{proof}[Proof of Theorem \ref{theorem/1/semimartingale}]
	Recalling $\nu(du) = c_\nu \delta_0(du) + f_\nu(u) du$, we have
	\begin{align*}
		\int_{-q}^0\int_{u+}^{t+u}  X_{s-}   dS_s \nu(du) 
		& =  c_\nu \int_{0+}^{t} X_{s-} dS_s + \int_{-q}^{0-} \int_{u+}^{t+u} X_{s-} dS_s f_\nu(u) du
		:= \mathrm{\bf I}+\mathrm{\bf II}.
	\end{align*}
	Since $X\in \cS^2$  and is hence locally bounded and progressively measurable, by Fubini's theorem, 
	exchanging the order of integration of $\mathrm{\bf II}$ gives
	\begin{align*}
		\mathrm{\bf II}
		& = \int_{(-q,t]}\left(\int \ind_{[-q,0)}(u) \ind_{(u,t+u]}(s)  X_{s-} f_\nu(u) du \right)  dS_s
		  = \int_{(-q,t]} F_\nu(t,s) X_{s-} dS_s,
	\end{align*}
	for $t\ge 0$, where the kernel $F_\nu$ is given by \ref{equation/1/semimartingale/kernelF}. 
	The computations for the $d\mu$ integral in \eqref{equation/1/model/CDGARCH_X} are exactly the same and the integral equation \eqref{equation/1/semimartingale/XwithFkernel} follows immediately.

	For the functional differential equation, first observe that $t \mapsto I_t$ in \eqref{equation/1/semimartingale/proof/Iasintegral} is Lipschitz and hence absolutely continuous. Hence $I_t  -I_0 = \int_{(0, t]} \frac{d}{dt}I_s ds$, where $\frac{d}{dt}I_t$ is given by \eqref{equation/1/semimartingale/proof/Iasintegralderivative},
	with $I_0 =0 $ since $F_\nu(0,s)=0$ for any $s$. The integral involving $F_\mu$ can be differentiated in exactly the same way. The functional differential equation follows immediately.

	Finally, since $S$ is of finite variation and $X$ is c\`adl\`ag, $X$ is a semimartingale with finite variations by Theorem I.4.31 of \citet{JacodShiryaev2013}. 
\end{proof}

\begin{proof}[Proof of Proposition \ref{proposition/1/semimartingale/pathofX}]
	\pf{(\ref{proposition/1/semimartingale/pathofX/jumps})} follows immediately from \eqref{equation/1/semimartingale/XSDDE}. 

	\pf{(\ref{proposition/1/semimartingale/pathofX/betweenjumps}).}
	On $(T_j, T_{j+1})$, $S_t = S(T_j) $, so by (\ref{proposition/1/semimartingale/pathofX/jumps}) of Proposition \ref{proposition/1/semimartingale/pathofX}, $X$ is continuous on $(T_j, T_{j+1})$. 
	With the normalization $f(-q) =0$, $\Delta \xi(X)_t = f_\nu(0)X_{t-}\Delta S_t$, so $\xi(X)$ is continuous on $(T_j, T_{j+1})$ as well. The rest of the proposition follows immediately.
\end{proof}

\begin{proof}[Proof of Proposition \ref{proposition/1/solution/CPPpositivity}]
		Suppose $X_t\ge x^-$ for all $t\in[-r,T]$,  for some $T\ge 0$ and let $T' := \inf\{t>T, \Delta S >0\}$. Since $S$ is a compound Poisson process, almost surely $T'>T$ and the interval $[T, T')$ is non-empty.
		Then by Proposition \ref{proposition/1/semimartingale/pathofX}(\ref{proposition/1/semimartingale/pathofX/betweenjumps}), $X$ is continuously differentiable in $[T, T')$ with derivative given by $\dot X_t = \eta - c_\mu X_t + \xi(X)_t$. 
		Since $\Delta X_t\ge 0$ whenever $X_{t-}\ge 0$, by iterating this argument, it suffices to show that $X_t \ge x^-$ for all $t\in[T,T')$.
		
		Let $T'' := \inf\{t>T, X_t< x^-\}$ and suppose for a contradiction that $T''\le T'$ with positive probability. Since $X$ is a.s. continuous at $T''$, necessarily $X_{T''} = x^-$ and $\dot X_{T''} <0$. But 
		\begin{align*}
			\frac{d}{dt}X_t\Big|_{t = T''} & \ge   \eta - c_\mu X_{T''} + \int_{-p}^0 X_{T''+u} f_\mu(u) du \ge 0
		\end{align*}
		almost surely, which contradicts our assumption. 		
\end{proof}

\subsection{Approximation by processes of finite activity}
\begin{proof}[Proof of Proposition \ref{proposition/1/solution/approximation}]
	\pf{(\ref{proposition/1/solution/approximation/Sn to S}).}
	From the construction of $S^n$ in \eqref{equation/1/solution/Sn without BM}, we have
	\begin{align*}
		S_u  - S_u^n = \int_{-r}^u \int_{0<|z|<\frac{1}{n}} z^2 N(dz,ds),
	\end{align*}
	which in non-decreasing in $u$. Fixing $t>0$, we have
	\begin{align*}
		\Em\left[\sup_{u\le t}|S_u-S^n_u|^2\right] 
		& =  \Em\left[\left(\int_{-r}^t\int_{0<|z|<\frac{1}{n}}z^2 N(dz, ds)\right)^2 \right]
		= \int_{-r}^t\int_{0<|z|<\frac{1}{n}}z^4 \Pi_L(dz) .
	\end{align*}
	Since $n\ge 1$, the integrand is dominated by $z^2$ which is $d \Pi_L$ integrable. By the dominated convergence theorem,
	$\Em\left[\sup_{u\le t}|S_u-S^n_u|^2\right] \to 0$ as $n\to \infty$. That is, $S^n$ approximates $S$ in each $\norm{\cdot}_t$, $t>0$. This clearly implies convergence in the ucp topology.

	\vspace{2mm}
	\pf{(\ref{proposition/1/solution/approximation/functional lipschitz}).}
	Let $Y$ and $Z$ be c\`adl\`ag processes in $\cS^2$,  then
	\begin{align*}
		|b(Y)_t - b(X)_t|
		& \le c_\mu |Y_t - Z_t| 
		+ \int_{t-p}^t f_\mu (u-t)|Y_u- Z_u| du
		+ \int_{t-q+}^t f_\nu (u-t)|Y_u- Z_u| dS_u
		   \\
		& \le \sup_{s\le t} |Y_s - Z_s| \left(c_\mu + \norm{f_\mu }_{L^1} + \int_{t-q+}^t f_\nu (u-t) dS_u\right).
	\end{align*}
	Since $f_\nu$ is assumed to be continuous and normalized to $f(-q)=0$, integrating by parts gives
	\begin{align*}
		\int_{t-q+}^t f_\nu (u-t) dS_u = S_t f_\nu(0)  - \int_{q+}^0 S_{t+u} df_\nu(u)\le S_t^* (f_\nu(0)+ \norm{f_\nu}_{L^1}),
	\end{align*}
	which implies that $b$ is functional Lipschitz. For each $b^n$, it suffices to carry through the same computation and observe that by construction, $S^n_t\le S_t$ for each $n\ge 1$ and $t>0$.

	\vspace{2mm}
	\pf{(\ref{proposition/1/solution/approximation/bn to b}).}
	From the definitions of $b^n$, for each $t\ge 0$, 
	\begin{align*}
		b(X)_t -  b^n(X)_t 
		=
		\int_{t-q+}^t f_\nu(u-t) X_{u-} dS_u - \int_{t-q+}^t f_\nu(u-t) X_{u-} dS_u^n.
	\end{align*}
	By the construction of $S^n$, we have
	\begin{align*}
		\left|b(X)_t -  b^n(X)_t  \right|
		\le  \sup_{u\in(t-q, t]} |f_\nu (u-t) X_{u-}| \big(S_t - S_{t}^n - (S_{t-q+} - S_{t-q+}^n)\big),
	\end{align*} 
	which converges to zero in $\norm{\cdot}_t$ for each $t$ and hence in ucp by Proposition \ref{proposition/1/solution/approximation}(\ref{proposition/1/solution/approximation/Sn to S}).
\end{proof}

\begin{proof}[Proof of Theorem \ref{theorem/1/solution/approximation}]
	The claim directly follows from Proposition \ref{proposition/1/solution/approximation} and Theorem V.15 of \cite{Protter2004}. More accurately, we invoke a trivial extension of Theorem V.15 of \cite{Protter2004} to the case with multiple driving semimartingales (see comments on page 257 of \cite{Protter2004}).  
\end{proof} 

\begin{proof}[Proof of Theorem \ref{theorem/1/solution/positivity}]
	For a given $S$ of the form \eqref{equation/1/solution/S without BM} satisfying Assumption \ref{assumptions/1/existence}(\ref{assumptions/1/existence/S}), let $(S^n)_n$ be as defined in \eqref{equation/1/solution/Sn without BM}. By Theorem \ref{theorem/1/existence}, we can set $X$ and $(X^n)_n$ to be unique solutions to \eqref{equation/1/semimartingale/XSDDE} and \eqref{equation/1/solution/approximation/sdden} driven by $S$ and $(S^n)_n$ respectively.

	By Theorem \ref{theorem/1/solution/approximation}, $X^n$ converges to $X$ in ucp, which trivially implies that for each $t>0$, $X^n_t \to X_t$ in probability and hence in distribution.  
	Furthermore, since each $S^n$ is a compound Poisson process by construction, by Proposition \ref{proposition/1/solution/CPPpositivity}, for each $n\ge 1$ and $t>0$, we have $X^n_t\ge  x^-$ with probability one, where $x^->0$ is defined in \eqref{equation/1/semimartingale/lowerbound}. 
	Finally, since $(-\infty, x^-)$ is open in $\Rm$, by the Portmanteau theorem of weak convergence (Theorem 2.1 \citet{Billingsley2013}), we have for each $t>0$,
	\begin{align*}
		\Pm(X_t <x^-)\le   \liminf_n \Pm(X_t^n<x^-) =0,
	\end{align*}   
	which completes the proof.
\end{proof}

\subsection{Moment bounds}\label{section/proofs/1/moments}

We precede the proof of Theorem \ref{theorem/1/moments/moment bound} with the following two lemmas.
\begin{lemma}[Lemma 8.1 - 8.2,  \citet{ItoNisio1964}]\label{lemma/1/stationarity/ItoNisio}
	Suppose $x,y:[0, \infty)\to \Rm_+$ are continuous functions, $\alpha>0$ and $\lambda_1>\lambda_2>0$. For every $0\le t<\infty$, 
	\begin{enumerate}[(a)]
	\item \label{lemma/1/stationarity/ItoNisio/1} if $x_t\le  x_0 - \lambda_1\int_{0}^t x_u du + \int_0^t y_u du$, then $x_t\le x_0 + \int_0^t e^{-\lambda_1(t-u)} y_u du$;

	\item \label{lemma/1/stationarity/ItoNisio/2} if  $x_t\le  \alpha + \lambda_2\int_{0}^t e^{-\lambda_1 (t-s)}x_u du$, then $x_t \le \alpha\lambda_1/({\lambda_1 - \lambda_2})$.\qed
	\end{enumerate}
\end{lemma}

\begin{lemma}\label{lemma/1/stationarity/cross}
	Suppose Assumptions \ref{assumptions/1/existence} hold, let $(X_t)_{t\ge 0}$ be the unique strong solution to \eqref{equation/1/semimartingale/XSDDE} with initial condition $\Phi$ and let $(\xi(X)_t)_{t\ge 0}$ be as defined in \eqref{equation/1/semimartingale/xi}. 
	For $n\in\{1,2\}$, we have
	\begin{align*}
		\Em [|X_t^{n-1}\xi(X)_t|]\le C_k \sup_{u\in [t-r, t]}\Em[|X_u|^n],\quad t>0, 
	\end{align*}
	where $C_1 = \norm{f_\mu}_{L^1} + \kappa_2 \norm{f_\nu}_{L^1}$ and $C_2 = \norm{f_\mu}_{L^1} +  \kappa_4 \norm{f_\nu}_{L^2}$. 

	\begin{proof}
		For the case of $n=2$, by Fubini's theorem and the Cauchy-Schwartz inequality, 
		\begin{align*}
			\Em[|X_t\xi(X)_t| ] 
			&\le \Em\left[|X_t| \int f_\mu(u-t)|X_u| du\right] + \Em\left[|X_t| \int f_\nu(u-t)|X_{u-}| dS_u\right]
			\\
			&\le \int f_\mu(u-t)\Em[|X_t| |X_u|  ]du 
			+ \Em [X_t^2]^{\frac{1}{2}} \Em\left[\left(\int f_\nu(u-t) |X_{u-}| dS_u\right)^2 \right]^{\frac{1}{2}}
			   \\
			& \le \int f_\mu(u-t)\Em[X_t^{2}]^{\frac{1}{2}} \Em [X_u^2]^{\frac{1}{2}}du
			+ \Em [X_t^2]^{\frac{1}{2}} \kappa_4 \left(\int f_\nu(u-t)^2 \Em[X_u ^2] du\right)^{\frac{1}{2}}
			   \\
			& \le \Big(\norm{f_\mu}_{L^1} + \kappa_4 \norm{f_\nu}_{L^2}\Big) \sup_{u\in[t-r,t]}\Em[X_u^2].
		\end{align*}
		The case of $n=1$ easily follows from the same computations.
	\end{proof}

\end{lemma}

\begin{proof}[Proof of Theorem \ref{theorem/1/moments/moment bound}]
\pf{(\ref{theorem/1/moments/moment bound/L1uniform}) and (\ref{theorem/1/moments/moment bound/L2uniform}).}
	The following proof holds for both $n=1$ and $n=2$, with different corresponding constants.
	Let $X$ be a positive solution to \eqref{equation/1/semimartingale/XSDDE} with $\eta>0$. For $n=2$, it follows from Ito's Lemma (\cite[Theorem I.4.57]{JacodShiryaev2013}) that (the $n=1$ case is trivial), 
		\begin{align*}
			X_t^{n} =X_0^n+ n\int_0^t X_s^{n-1} \Big(\eta - c_\mu X_s + \xi(X)_s\Big) ds
			+ \sum_{0<s\le t} \left\{X_s^n - X_{s-}^n \right\}
			\numberthis\label{equation/1/proofs/1/stationarity/X^n by Ito's formula}
		\end{align*}
		where $\Delta X_t = c_\nu X_{t-} \Delta S_t$ and
		\begin{align*}
			\sum_{0<s\le t} \left\{X_s^2 - X_{s-}^2 \right\}
			=
			\sum_{0<s\le t} \left\{ (X_{s-} + c_\nu X_{s-} \Delta S_s)^2 - X_{s-}^2 \right\}
			=
			\sum_{0<s\le t} \left\{ X_{s-}^2 \big(2c_\nu \Delta S_s + c_\nu^2 (\Delta S_s)^2\big) \right\}.
		\end{align*}
		Let $E_n(t) := \Em[X_t^{n}]$ and put $K_1 := \kappa_2 c_\nu$ and $K_2:= \kappa_2 c_\nu + \frac{1}{2}\kappa_4c_\nu^2$.
		From \eqref{equation/1/proofs/1/stationarity/X^n by Ito's formula}, we have
		\begin{align*}
		E_n(t)  
		&= E_n(0) 
			+n \Em\left[ \int_0^t X_s^{n-1}
			 \Big( \eta  + (-  c_\mu+ K_n) X_s  \Big)  ds\right]
			+n \Em\left[ \int_0^t X_s^{n-1} \xi(X)_s ds 
				\right].
			\numberthis\label{equation/1/proofs/1/stationarity/pqUI/E_n(t)}
		\end{align*}
		Let $C_1, C_2$ be given as in Lemma \ref{lemma/1/stationarity/cross} and suppose
		\begin{align*}
			\lambda_n:= c_\mu - K_n - C_n >0.\numberthis\label{equation/1/proofs/1/stationarity/pqUI/lambda}
		\end{align*}
		An exercise in calculus gives $\sup_{x\ge 0} 2x\Big(\eta - \frac{1}{2}\lambda_2 x\Big) = \eta^{2}/\lambda_2 =:a_2$ and 
		$\sup_{x\ge 0} \Big(\eta - \frac{1}{2}\lambda_1 x\Big) = \eta =:a_1$ 
		Then for all $X_s\ge 0$,
		\begin{align*}
			n X_s^{n-1}\Big(
			\eta + \frac{1}{2}\big(-c_\mu + K_n +  C_n \big)
			 X_s \Big) < a_n,
		\end{align*}
		which we rearrange to get a bound for the integrand in the first integral in \eqref{equation/1/proofs/1/stationarity/pqUI/E_n(t)}:
		\begin{align*}
			&n X_s^{n-1}
			\Big(\eta + \big( - c_\mu+K_n\big) X_s \Big)
			< a_n - \frac{n}{2}\big( c_\mu - K_n + C_n  \big) X_s^{n}.
			\numberthis\label{equation/1/proofs/1/stationarity/pqUI/bound_on_eta_term}
		\end{align*}
		For the second integral of \eqref{equation/1/proofs/1/stationarity/pqUI/E_n(t)},  Lemma \ref{lemma/1/stationarity/cross} gives the bound
		\begin{align*}
			&n\Em\left[X_s^{n-1} \xi(X)_s\right]
			\le 
			n C_n \sup_{u\le s}\Em[X_u ^n].
		\end{align*}
		Combining this with \eqref{equation/1/proofs/1/stationarity/pqUI/E_n(t)} and \eqref{equation/1/proofs/1/stationarity/pqUI/bound_on_eta_term} 
		and writing $\bar E_n(s) = \sup_{u\le s} E(u)$,
		we have 
		\begin{align*}
			E_n(t)\le E_n(0)
			- \lambda'_n
			\int_0^t E_n(s) ds
			 +  \int_0^t\left( a_n+ \lambda_n ''\bar E_n(s)\right) ds,
		\end{align*}
		where the constants $\lambda'_n$ and $\lambda_n''$ are given by $\lambda_n':  =  \frac{1}{2}n(c_\mu - K_n + C_n )$ and $\lambda_n'' := nC_n$.
		Our assumed condition \eqref{equation/1/proofs/1/stationarity/pqUI/lambda} gives
		$\lambda'_n - \lambda''_n  = \frac{1}{2} n\lambda_n >0$.
		By Lemma \ref{lemma/1/stationarity/ItoNisio} (\ref{lemma/1/stationarity/ItoNisio/1}), 
		\begin{align*}
			E_n(t)\le  E_n(0) + \int_0^t e^{-\lambda'_n(t-s)}  \left(a_n+\lambda_n''\bar E_n(s)\right) ds. 
			\numberthis\label{equation/1/proofs/1/stationarity/pqUI/E_n(t)afterfirstGronwall}
		\end{align*}
		Since $\bar E$ is non-decreasing and $\lambda_n'>0$, an integration by parts shows 
		\begin{align*}
			\int_0^t e^{-\lambda_n'(t-s)} & \left(a_n+\lambda_n''\bar E_n(s)\right) ds
			= \frac{(1-e^{-\lambda_n' t})(a_n + \lambda_n''\bar E_n(0))}{\lambda_n'}
			+\frac{\lambda_n''}{\lambda_n'}\int_0^t \left(1-e^{-\lambda_n'(t-s)}\right)  d\bar E_n(s)
			   \\
			& = \frac{(1-e^{-\lambda_n' t})(a_n + \lambda_n''\bar E_n(0))}{\lambda_n'}
			+\frac{\lambda_n''}{\lambda_n'} (\bar E_n(t)- \bar E_n(0)) -e^{-\lambda_n't}\int_0^t e^{\lambda_n's} d\bar E_n(s).
		\end{align*}
		The last expression is an non-decreasing function of $t$.
		Hence taking supremums in \eqref{equation/1/proofs/1/stationarity/pqUI/E_n(t)afterfirstGronwall}, we have
		\begin{align*}
			\bar E_n(t): = \sup_{u\in [0,t]} E(u) 
			& \le E_n(0) + \sup_{u\in [0,t]} \int_0^u e^{-\lambda_n'(u-s)}\left(a_n+\lambda_n''\bar E_n(s)\right) ds
			   \\
			& \le \Em[X_0^n] +  \frac{a_n}{\lambda_n'}
			+ \lambda_n'' \int_0^t e^{-\lambda_n'(t-s)} \bar E_n(s)ds.
		\end{align*}
		By Lemma \ref{lemma/1/stationarity/ItoNisio} (\ref{lemma/1/stationarity/ItoNisio/2}), since $\lambda_n'>\lambda_n''>0$ and $\lambda_n' - \lambda_n'' = \frac{1}{2}n\lambda_n$ for all $t\ge 0$, we have
		\begin{align*}
			\bar E_n(t)
			\le \left(\Em[X_0^n] + \frac{a_n}{\lambda_n'}\right) \frac{\lambda_n'}{\lambda_n'-\lambda_n''}
			= \frac{2a_n/n}{c_\mu - K_n - C_n}  + \Em[X_0^n]\frac{ c_\mu - K_n + C_n }{c_\mu - K_n - C_n}  <\infty.
		\end{align*}
		The Theorem follows immediately.
\end{proof}

\begin{remark*}
	The above arguments are partly adapted from \cite{BaoYinYuanEtAl2014}, but we expand on their arguments and derive explicit bounds and constants to have explicit bounds for the first and second moments. 
\end{remark*}


	\subsection{Moment processes}\label{section/proofs/mean}
	\begin{proof}[Proof of Theorem \ref{theorem/1/moments/mean}]\pf{(\ref{theorem/1/moments/meanFDE}).}
		Since $X\in \cS^2$, the stochastic integral process $X \cdot \td S$ is a true martingale. 
		Taking expectation of the equation \eqref{equation/1/semimartingale/XSDDE}, we get 
		\begin{align*}
			m(t) & =  m(0) + \int_{0}^t \Big(\eta - c_\mu m(s) + \Em[\xi(X)_s]\Big)ds 
			+ \kappa_2c_\nu\int_0^t m(s) ds
			   \\
			& = m(0) + \int_{0}^t \left(\eta - c_0 m(s) 
						+ \int_{s-p}^s f_\mu(u-s) m(u) du + \kappa_2\int_{s-q}^s f_\nu(u-s) m(u) du\right)ds.
		\end{align*}
		Recalling the definitions of $f$ and $c_0$ before Theorem \ref{theorem/1/moments/mean}, we have the integral equation
		\begin{align*}
			m(t)& = m(0) + \int_{0}^t \left(\eta - c_0 m(s) 
						+ \int_{-r}^0 f(u) m(s+u) du \right)ds.
		\end{align*}
		From Theorem \ref{theorem/1/existence}, we know that the function $t\mapsto m(t)$ is c\`adl\`ag and hence locally bounded. Since $f$ is integrable, for any $t_1\le t_2$, by the dominated convergence theorem, 
			\begin{align*}
				&\left|\int_{-r}^0 f(u) \Big(m(t_2+u)-m(t_1+u)\Big) du \right|
				\le \int _{t_1-r}^{t_2} \big| f(u-t_2)- f(u-t_1)\big|  |m(u)| du
					\\
				&\qquad \le \left(\sup_{u\in [t_1 -r, t_2]} |m(u)|\right) \int \big|f(u-t_2)- f(u-t_1)\big| du
				 \to 0, \quad\text{as } {|t_2-t_1|\to 0},
			\end{align*}
			so the function $t\mapsto \int_{-r}^0 f(u) m(t+u) du$ is continuous.
	 	Furthermore, 
		\begin{align*}
			|m(t_2) - m(t_1)| \le |t_2 - t_1|
			\left( \eta + c_0 \sup_{u\in[t_1, t_2]} |m(u)| + \norm{f}_{L^1} \sup_{u\in[t_1-r, t_2] }|m(u)|\right),
		\end{align*}
		so the function $t\mapsto m(t)$ is continuous as well. Therefore $t\mapsto m(t)$ is continuously differentiable and the differential equation follows.

		\vspace{1.2mm}\pf{(\ref{theorem/1/moments/meanRE}).} The proof is adapted from Section 6.1 of \citet{HaleLunel2013}.
		Put $M:= \eta/(c_0 - \norm{f}_{L^1})$ and $\td m := m - M$ and $\td \phi := \phi - M$, then clearly $\td m$ is the solution to the linear delay equation
		\begin{align*}
			 \td m'(t)& = -c_0 \td m(t) + \int_{-r}^0  \td m(t+u) f(u) du,\quad t\in[0,\infty)
			\numberthis\label{equation/proofs/1/moments/Xnorm_mean_fde},
		\end{align*}
		with initial condition $\td m = \td \phi$ on $[-r,0]$.
		With $\zeta$ defined in \eqref{equation/1/moments/meanNBV}, we can rewrite \eqref{equation/proofs/1/moments/Xnorm_mean_fde} into
		\begin{align*}
			\td m'(t) =  \int_0^r \td m(t-u) d \zeta(u). 
		\end{align*}
		For $0\le t\le r$, we can separate the initial condition in \eqref{equation/proofs/1/moments/Xnorm_mean_fde} to obtain
		\begin{align}
			\td m'(t) = \int_0^t  \td m(t-u) f(u)du+ \int_t^r  \td \phi(t-u)f(u)du.\label{equation/proofs/1/moments/Xnorm_mean_separateinitial}
		\end{align}
		Now, since $\zeta$ is by construction constant for $t\ge r$, \eqref{equation/proofs/1/moments/Xnorm_mean_separateinitial} holds for $t>r$ also. Integrating by parts, we obtain a renewal equation for $\td m'$:
		\begin{align}
			\td m'(t) & = \int_0^t  \td m'(t-u) \zeta(u)du + g(t), \quad t\in[0,\infty),\label{equation/proofs/1/moments/Xnorm'_RE}
		\end{align}
		with initial condition $\td m(0)=\td \phi(0)$, where $g(t):=\zeta(t)\td m(0) +\int_t^r  \td \phi(t-u) f(u)du$.
		Integrating \eqref{equation/proofs/1/moments/Xnorm'_RE} and changing the order of integration, we obtain
		\begin{align*}
			&\td m(t)- \td m(0)  =  \int_0^t \int_0^s \zeta(u)  \td m'(s-u) du\ ds + \int_0^t g(s) ds
				\\
			&\qquad= \int_0^t \zeta(u) \int_u^t   \td m'(s-u) ds\ du + \int_0^t g(s) ds
			= \int_0^t \zeta(u) \td m(t-u)du -\int_0^t \zeta(u)\td m(0) du + \int_0^t g(s) ds.
		\end{align*}
		Changing variables $u\mapsto t-u$, we arrive at a renewal equation for $\td m$:
		\begin{align*}
			\td m(t) & 
				 = \int_0^t  \zeta(t-u)\td m(u) du + \td h(t), \quad t\in[0,\infty),\label{equation/proofs/1/moments/Xnorm_RE}\numberthis
		\end{align*}
		with initial condition $\td m(0)=\td \phi(0)$. The forcing function, $\td h$, given by 
		\begin{align*}
			\td h(t) :& =  \td \phi(0) -\int_0^t \zeta(u)\td \phi(0) du + \int_0^t g(s) ds
				= \td \phi(0) + \int_{-r}^0 (\zeta(t+u) - \zeta(u) )\td \phi(-u) du, \numberthis\label{equation/proofs/1/moments/Xnorm_h}
		\end{align*}
		is Lipschitz continuous on $[0,r]$ and constant for $t\ge r$ \citep[p.18]{DiekmannGilsLunelEtAl2012}. Since $\zeta(-r) = -M$, substituting $\td m=m+ \eta/\zeta(r)$ and $\td \phi=  \phi + \eta/\zeta(r)$ back into \eqref{equation/proofs/1/moments/Xnorm_RE} and \eqref{equation/proofs/1/moments/Xnorm_h} completes the computations. 
	\end{proof}

	\begin{proof}[Proof of Theorem \ref{theorem/1/moments/stationarymean}]
	\pf{(\ref{theorem/1/moments/stationarymean/limitingmean}).}
		Let $M$, $\td m$ and $\td \phi$ be as defined in the proof of Theorem \ref{theorem/1/moments/mean} (\ref{theorem/1/moments/meanRE}). 
		The {\it characteristic function} $\Delta$ of \eqref{equation/proofs/1/moments/Xnorm_mean_fde} defined in \eqref{equation/1/preliminaries/dde/characteristicfunction} is given by
		\begin{align*}
			\Delta(z) = z + c_0 - \int_{-r}^0 e^{z u} f(u)du. 
		\end{align*}
		It's clear from \eqref{equation/1/preliminaries/dde/asyExpansion} that if $\Delta(z)$ is root free in the right half-plane $\{z|\Re z\ge 0\}$, then all solutions $\td m$ of the functional differential equation \eqref{equation/proofs/1/moments/Xnorm_mean_fde} converge to zero exponentially fast as $t\to\infty$. 

		For sufficiency, it is enough to show that $c_0 > \norm{f}_{L^1}$ implies $\Delta(z)\ne0$ for any $z$ with $\Re z\ge 0$.  Let $z = \alpha + i \beta$ where $\alpha\ge 0$. Then  the real part of $\Delta$  can be written as
		\begin{align*}
			\Re \Delta(z)& =  \alpha +c_0 - \int_{-r}^0 e^{\alpha u} \cos(\beta u) f(u)du
		\end{align*}
		Since $e^{\alpha u}$ and $\cos(\beta u)$ are no greater than 1 on $[-r,0]$, we have 
		$\Re \Delta(z)	\ge \alpha + c_0 - \int_{-r}^0 f(u) du >0,$
		whenever $c_0 > \norm{f}_{L^1}$, so $\Delta$ is root free on $\{z|\Re z\ge 0\}$.
		For necessity, the expansion \eqref{equation/1/preliminaries/dde/asyExpansion} implies that 0 is the only possible limit of $\td m(t)$, which gives the uniqueness of $m$ as a limiting mean. Since we require this limit to be positive, necessarily we require $c_0>\norm{f}_{L^1}$.

		\vspace{2mm}
		\pf{(\ref{theorem/1/moments/stationarymean/stationarymean}).}
		Suppose that $ \phi\equiv M$ where $M$ is defined in \eqref{equation/1/moments/Mx}, and assume $c_0>\norm{f}_{L^1}$ so that $M>0$. Then $\td \phi$ is identically zero on $[-r,0]$, and the function $h$ defined in \eqref{equation/proofs/1/moments/Xnorm_h} is identically zero on $[-r,\infty)$. From \eqref{equation/proofs/1/moments/Xnorm_RE}, the centered mean process $\td m(\cdot)$ satisfies satisfies the homogeneous renewal equation 
		\begin{align*}
			\td m(t) = \int_0^t \zeta(u) \td m(t-u) du.
		\end{align*}
		Applying the representation in Theorem 2.12 of \citet{DiekmannGilsLunelEtAl2012} shows that the only solution to this renewal equation is $\td m(t)=0$ for all $t\in[0,\infty)$. This gives $m(t)= M$ for all $t\in[-r,\infty)$.
		Conversely, suppose that for all  $t\in[0,\infty)$, $m(t)=M$ for some positive $M$. Then \eqref{equation/1/moments/meanFDE} gives
		$0= \eta + M \left(-c_0+ \norm{f}_{L^1}\right)$,
		which implies that $M$ is uniquely given by \eqref{equation/1/moments/Mx} and $c_0> \norm{f}_{L^1}$. Recall that the delay equation \eqref{equation/1/moments/meanFDE} has a unique solution once the initial condition $ \phi$ is fixed. Therefore the solution $m \equiv M$ for all $t\ge 0$ then corresponds uniquely to the initial condition $\phi\equiv M$ on $[-r,0]$, and the proof is complete. 
	\end{proof}


\begin{proof}[Proof of Theorem \ref{theorem/1/moments/weakdependence}]
	Let $Z$ be an $\cF_t$ measurable random variable with $\Em[Z^2 X^2_t]<\infty$ for any $t>0$. Since $X$ has finite variations, for any $t>0$ and  $u>r$, we have
	\begin{align*}
		 X_{t+u} & =  Z X_{t+r} + \int_{t+r+}^{t+u} Z dX_s
		= Z X_{t+r} + \int_{t+r+}^{t+u} Z
		\left\{\Big(\eta - c_0X_s + \xi(X)_s \Big)ds + c_\nu X_{s-} d\td S_s
		\right\}.
	\end{align*}
	Taking expectations and using Fubini's theorem gives
	\begin{align*}
		\Em[Z X_{t+u}] 
		& = \Em\left[Z X_{t+r}\right] 
		+ \eta (u-r) \Em[Z] - c_0 \int_{t+r}^{t+u} \Em[Z X_s] ds 
		+ \int_{t+r}^{t+u} \int_{s-r}^s \Em[Z X_u]f(u-s) du ds
			\\
		& \hspace{1cm} + \int_{t+r}^{t+u} \Em\left[Z\int_{s-q+}^s X_{u-}f_\nu(u-s) d\td S_u \right]ds
		+ c_\nu \Em\left[\int_{t+r+}^{t+u} Z X_{s-} d\td S_s\right].
	\end{align*}
	Since $Z$ is $\cF_t$ measurable and hence $\cF_{s-q+}$ measurable for any $s\ge t+r$, the two stochastic integrals in the last expression  have zero expectation.
	Therefore
	\begin{align*}
		\Em[Z X_{t+u}] 
		& = \Em\left[Z X_{t+r}\right] 
		+\eta (u-r) \Em[Z] - c_0 \int_{t+r}^{t+u} \Em[Z X_s] ds
		+ \int_{t+r}^{t+u} \int_{-r}^0 \Em[Z X_{s+u}]f(u) du ds,
	\end{align*}
	from which we obtain a functional differential equation,
	\begin{align*}
		\frac{d}{du}\Em[Z X_{t+u}]  = \eta \Em[Z] - c_0 \Em[Z X_{t+u}] + \int_{-r}^0 \Em[Z X_{u+u}] f(u) du ds.
	\end{align*}
	Since we assumed $c_0>\norm{f}_{L^1}$, by the same argument as in the proof of Theorem \ref{theorem/1/moments/stationarymean} (\ref{theorem/1/moments/stationarymean/stationarymean}), we have
	\begin{align*}
		\Em[Z X_{t+u}]  \to \frac{\eta \Em[Z]}{c_0 - \norm{f}_{L^1}} = M\Em[Z] 
	\end{align*}
	exponentially fast.
\end{proof}


\begin{proof}[Proof of Theorem \ref{theorem/1/moments/price}]

\pf{(\ref{theorem/1/moments/price/statioanrycovariance})}
	Since $\td Y_t = \int_{t-1+}^{t} \rt{X_{s-}} dL_s$ and $L$ has zero mean, we have
	$\Em[\td Y_t]=0$
	and $\Cov(\td Y_t, \td Y_{t+u}) = \kappa_2\int \ind_{[t-1, t]}(s)\ind_{[t+u-1, t+u]}(s)\Em[ X_s]  ds=\kappa_2 M (1-u)_+$.

\vspace{2mm}
\pf{(\ref{theorem/1/moments/price/squaredreturns}).}
	Write $\kappa_1:= \int_{\Rm_0} z \Pi_L(dz)$. Since 
	$dY_t = -\kappa_1 \rt{X_{t}} dt + \int_{\Rm_0} \sqrt{X_{t-}}z \td N(dz, dt)$, by Ito's lemma, it holds that
	$dY_t^2  = -2Y_{t} \kappa_1 \rt{X_{t}} dt +  Y_t^2 - Y_{t-}^2$,
	where 
	\begin{align*}
		Y_t^2 -Y_{t-}^2 =\int_{\Rm_0} \left((Y_{t-} + \rt{X_{t-}}z )^2  -Y_{t-}^2\right) N(dz,dt)
		= \int_{\Rm_0} \left( 2 Y_{t-}\rt{X_{t-}}z + X_{t-}z^2 \right) N(dz,dt).
	\end{align*}
	Then, since $\td Y_{t+u}^2 = (Y_{t+u} - Y_{t+u-1})^2 = Y_{t+u}^2 - Y_{t+u-1}^2 - 2Y_{t+u-1}(Y_{t+u} - Y_{t+u-1})$, we have
	\begin{align*}
		\td Y_{t+u}^2 & = \int_{t+u-1+}^{t+u} 2Y_{s-} \sqrt{X_{s-}}  d L_s
		+ \int_{t+u-1+}^{t+u}\int_{\Rm_0} X_{t-} z^2 N(dz,dt)
		- 2Y_{t+u-1}\int_{t+u-1+}^{t+u} \sqrt{X_{s-}}dL_s.
	\end{align*}	
	Now suppose $u>1$ so that $\td Y_t$ is $\cF_{t+u-1}$ measurable. Taking expectations, we obtain
	\begin{align*}
		\Em[\td Y_t^2 \td Y_{t+u}^2] = \kappa_2\int_{t+u-1}^{t+u} \Em[\td Y_t^2 X_{s}] ds. 
	\end{align*}
	By Theorem \ref{theorem/1/moments/weakdependence} and Theorem \ref{theorem/1/moments/price} (\ref{theorem/1/moments/price/statioanrycovariance}), $\Em[\td Y_{t}^2X_u]\to \kappa_2 M^2$ exponentially fast as $u\to \infty$, i.e. there exists constants $C$ and $T,\lambda>0$ such that 
	$\left|\Em[\td Y_{t}^2X_u]- \kappa_2 M^2\right| \le C e^{-\lambda u},$ for all $u>T$. 
	Therefore
	\begin{align*}
		&\left|\Em[\td Y_t^2 \td Y_{t+u}^2]  - \kappa_2^2 M^2 \right|
		 \le \kappa_2 \int_{t+u-1}^{t+u}\left| \Em[\td Y_t^2 X_s]  - \kappa_2 M^2 \right| ds
		  \le \kappa_2 \int_{t+u-1}^{t+u} C e^{-\lambda s} ds= \frac{\kappa_2 C e^{1- \lambda t}}{\lambda}e^{-\lambda u},
	\end{align*}
	for all $u>T$, which finishes the proof.
\end{proof}

\section*{Acknowledgment}
The author's research is supported financially by the Australian Government Research Training Program (AGRTP). 
The author wishes to  thank Prof. Boris Buchmann and Prof. Ross Maller for their help and support, as well and Prof. William Dunsmuir for all the insightful discussions. 


\bibliographystyle{plainnat1}
\bibliography{../../Ref/Zotero/CDGARCHpaper}
\end{document}